%% file: REG.tex
\documentclass[hidelinks,onefignum,onetabnum]{siamart220329}

\input{siam_shared}

\ifpdf
\hypersetup{
  pdftitle={Multivariate Polynomial Regression of Euclidean Degree Extends the Stability for Fast Approximations of Trefethen Functions},
  pdfauthor={S. K. Thekke Veettil, Y. Zheng, U. H. Acosta, D. Wicaksono, and M. Hecht}
}
\fi

\usepackage{algorithm}
\usepackage{algpseudocode}
\usepackage{caption}
\usepackage{subcaption}

\usepackage{amsfonts,amsmath,amssymb}
\usepackage{mathptmx}
\usepackage{protosem}
\usepackage[numbers,sort&compress]{natbib}

\newcommand{\R}{\mathbb{R}}
\newcommand{\C}{\mathbb{C}}
\newcommand{\N}{\mathbb{N}}

\newcommand{\Cheb}{\mathrm{Cheb}}
\newcommand{\Oc}{\mathcal{O}}
\newcommand{\ee}{\varepsilon}
\newcommand{\lo}{\longrightarrow}
\newcommand{\li}{\left}
\newcommand{\re}{\right}

\newcommand{\rank}{\mathrm{rank}}
\newcommand{\cond}{\mathrm{cond}}

\newtheorem{experiment}{\sc Experiment}
{\bf}{\it}

\begin{document}

\maketitle

\begin{abstract}
  We address classic multivariate polynomial regression tasks from a novel perspective resting on the notion of general polynomial $l_p$-degree, with total, Euclidean, and maximum degree being the centre of considerations. While ensuring stability is a theoretically known and empirically observable limitation of any computational scheme seeking for fast function approximation,
  we show that choosing Euclidean degree resists the instability phenomenon best. Especially, for a class of analytic functions, we termed Trefethen functions, we extend recent argumentations that suggest this result to be genuine. We complement the novel regression scheme, presented herein, by an adaptive domain decomposition approach
  that extends the stability for fast function approximation even further.
\end{abstract}



\section{Introduction}

Classic $1$-dimensional (1D) polynomial interpolation goes back to Newton, Lagrange, and others, see, e.g.,\cite{LIP}. Its generalisation to regression tasks was
mainly proposed and developed by Gau\ss, Markov, and Gergonne \cite{gergonne1974application,stigler1974gergonne}
and is omnipresent in mathematics and computing till today.

Given a continuous multivariate function $f : \Omega \lo \R$, on the $m$-dimensional hypercube $\Omega = [-1,1]^m$, $m \in \N$ a set of data points $P \subseteq \Omega$, function values $F = (f(p))_{p \in P} \in \R^{|P|}$ and a polynomial basis $\{q_\alpha\}_{\alpha \in A}$, $A\subseteq \N^m$, e.g. the canonical $q_\alpha = x^\alpha =x_1^{\alpha_1}\cdots x_m^{\alpha_m}$,
we consider the case of regular regression tasks, deriving the polynomial coefficients $C=(c_\alpha)_{\alpha \in A} \in \R^{|A|}$
in the classic least square sense
\begin{equation}\label{eq:REG}
  C =\mathrm{argmin}_{X \in \R^{|A|}}\|R_{A}X-F\|^2\,.
\end{equation}
The regression matrix  $R_A =(q_{\alpha}(P_i))_{i=1,\dots,|P|, \alpha \in A}\in \R^{|P| \times |A|}$ (see section~\ref{sec:NL_REG}) is assumed to be of full $\rank \, R_A =|A| \leq |P|$, and the multi-index set $A \subseteq \N^m$ generalises the notion of 1D polynomial degree to multivariate $l_p$-degree, i.e,
\begin{equation}\label{eq:LP}
  A = A_{m,n,p}:=   \{\alpha \in \N^m : \|\alpha\|_p \leq n\}\,, \quad p >0\,.
\end{equation}
The particular cases of \emph{total degree} $A = A_{m,n,1}$, \emph{Euclidean degree} $A= A_{m,n,2}$, and \emph{maximum degree} $A= A_{m,n,\infty}$
play a crucial role for the polynomial approximation power. Due to \cite{Lloyd2} (see also section~\ref{sec:Notation}), the sizes of the sets scale polynomially, sub-exponentially, and exponentially with dimension, respectively:
\begin{equation}\label{eq:size}
  |A_{m,n,1}| = \binom{m+n}{n} \in \Oc(m^n)\,, \,\,\, |A_{m,n,2}|\approx \frac{(n+1)^m }{\sqrt{\pi m}} \li(\frac{\pi \mathrm{e}}{2m}\re)^{m/2} \in o(n^m)\,, \,\,\, |A_{m,n,\infty}| = (n+1)^m\,.
\end{equation}

The setup in Eq.~\eqref{eq:REG} is a classic for the total or maximum degree choices. Motivated by recent works of \cite{Lloyd2,converse,MIP}, we theoretically and empirically show that for regular functions including a class of analytic functions we term \emph{Trefethen functions} these choices are sub-optimal compared to the Euclidean degree.

\subsection{Trefethen functions}\label{sec:TREF}

We term the class of continuous functions $f \in T(\Omega,\R) \subseteq C^0(\Omega,\R)$ that possess an absolute convergent Chebyshev series expansion on $\Omega$, \linebreak
$f  =\sum_{\alpha \in \N^m}c_\alpha \prod T_{\alpha_i}\in \Pi_{A_{m,n,p}}$, \emph{Trefethen functions} if $f$ can be \emph{analytically  extended}
to the (open, unbounded) \emph{Trefethen domain}
of radius $\rho = h + \sqrt{1 + h^2}$, $h>0$
\begin{equation}\label{Tdomain}
   N_{m,\rho} = \li\{ (z_1,\dots,z_m) \in \C^m :  (z_1^2 + \cdots  + z_m^2) \in E_{m,h^2}^2 \re\}\,.
 \end{equation}
 Here,
 $E_{m,h^2}^2$ denotes the \emph{Newton ellipse} with foci $0$ and $m$ and leftmost point $-h^2$.
 A prominent example of a Trefethen function is given by the \emph{Runge function}, $f(x) = 1/(1 + r\|x\|^2)$, $r >1$, \cite{runge}.

In \cite{Lloyd2} Trefethen proved an upper bound on the convergence rate for truncating the Trefethen function
$\mathcal{T}_{A_{m,n,p}}(f)  =\sum_{\alpha \in A_{m,n,p}}c_\alpha \prod T_{\alpha_i}\in \Pi_{A_{m,n,p}}$
to the polynomial space $\Pi_{A_{m,n,p}}$:
\begin{equation}\label{Rate}
  \| f - \mathcal{T}_{A_{m,n,p}}(f)\|_{C^0(\Omega)}  = \li\{\begin{array}{ll}
                                                                      \Oc_\ee(\rho^{-n/\sqrt{m}})  &\,, \quad p =1 \\
                                                                      \Oc_\ee(\rho^{-n}) &\,, \quad p =2\\
                                                                      \Oc_\ee(\rho^{-n})  &\,, \quad p =\infty
                                                                      \end{array}
\re.\,,
\end{equation}
where  $g \in \Oc_\ee(\rho^{-n})$ if and only if $g \in \Oc((\rho-\ee)^{-n})$ $\forall \ee >0$.

This suggests that interpolation with respect to \emph{Euclidean $l_2$-degree} can reach faster convergence rates than with \emph{total $l_1$-degree}, and performs compatibly with the maximum degree.
This has been validated and argued to be genuine~\cite{Lloyd2,converse,MIP}. Here, we adapt the non-tensorial interpolation scheme \cite{MIP,minterpy} for regression tasks.
Compared to the choice of $l_\infty$-degree, the resulting smaller
regression matrices $R_{A,P} \in \R^{|P| \times |A|}$ given in Eq.~\eqref{eq:size}, Eq.~\eqref{eq:RA}, do not only lower storage amount and runtime costs, but more crucially are
better conditioned, which is the essential limiting factor for regression schemes.


\subsection{The range of fast and stable approximations - Theoretical limitations of regression schemes}
\label{sec:limit}

Platte et al. \cite{platte:2011} have proven that regardless of the choice of the (non-polynomial) basis functions or approximation method,
 guaranteeing stability (sub-exponen- tially growing condition numbers $\cond (R_{A,P}) \in o(r^n)$, $r>1$) of an approximation scheme with fast (exponential approximation rates) for analytic functions on equispaced data is impossible. In other words: Any regression scheme that seeks for fast approximations of analytic functions $ f: \Omega \longrightarrow \R$ sampled on
 equispaced grids $G \subseteq \Omega$ comes with the cost of becoming rapidly unstable. Moreover, approximation in equispaced grids is highly sensitive to Runge's and Gibbs' phenomenon \cite{runge,hewitt1979gibbs,dimarogonas1996vibration},
demanding a high-order function expansion in order to be suppressed.

While regression tasks for functions sampled on regular, equispaced grids $G \subseteq \Omega$ is very often appearing in practice, even samples from
"real world" experiments or
other (random) sampling choices can not be expected to circumvent the instability phenomenon in general. Especially for equispaced data,
the range of fast and stable (least square) regression is limited to a class of well-behaving functions $T(\Omega,\R)\subseteq C^0(\Omega,\R)$.

Contributing to broaden the function space $T(\Omega,\R)\subseteq C^0(\Omega,\R)$ for which a close (machine precision) approximation can be reached is the scope of this article.

%
%
%
%
%
%

\subsection{Contribution}\label{sec:limits}
Based on our recent results on multivariate interpolation \cite{PIP1,PIP2,MIP,IEEE} we, here, deliver three essential contributions:

\begin{enumerate}
  \item[i)] \emph{Newton-Lagrange regression of general $l_p$-degree:} We propose a novel polynomial regression scheme that rests on our recent extension of classic 1D  Newton-Lagrange interpolation for multi-dimensional, non-tensorial, unisolvent interpolation nodes $P_A = \{p_\alpha\}_{\alpha \in A} \subseteq \Omega$ (generated by Leja ordered Chebyshev-Lobatto nodes, Definition~\ref{def:EA}) with respect to general $l_p$-degree, i.e.,  $A = A_{m,n,p}$, $p >0 $. The resulting regression matrix $R_A$, appearing in Eq.~\eqref{eq:REG}, is  generated by evaluating  a non-tensorial, polynomial Lagrange basis $\{L_{\alpha}\}_{\alpha \in A}$, $L_{\alpha}(p_\beta) = \delta_{\alpha,\beta}$,
 $ \forall \alpha, \beta \in A$, $p_\beta \in P_A$ in the data points $P \subseteq \Omega$, i.e.,
  \begin{equation}\label{eq:RA}
  R_{A,P} = \big (L_{\alpha}(p_i)\big)_{i= 1,\ldots,|P|, \alpha \in A } \in \R^{|P|\times |A|}\,.
  \end{equation}
   Compared to the tensorial, $l_\infty$-degree choices, on top of memory  and runtime reductions given by the smaller size of $R_{A,P}$, demonstrations in Section~\ref{sec:Num}
   validate that the Euclidean degree regression significantly extends the stability range without loosing approximation power, as it is the case for the total $l_1$-degree regression.
  \item[ii)] \emph{Approximation theory:} In Theorem~\ref{theo:APP} we rephrase the classic Lebesgue inequality of polynomial regression, see, e.g.,\cite{brutman2} by relating the fast approximation rates
  $$\|f - Q_{f,A_{m,n,p}}\|_{C^0(\Omega)} = \Oc(\rho^{-n})\,, \quad  \rho >1\,,$$
   that can be reached for the (Euclidean ) $l_p$-degree \emph{interpolation of Trefethen functions} in proper interpolation nodes $P_A\subseteq \Omega$ proposed in our former work \cite{MIP}.
   That is, we estimate the approximation error of the regressor $Q_{f,P,A}$ of $f$ in $P\subseteq \Omega$ by the approximation error of the (unknown)
   interpolant $Q_{f,A}$:
\begin{equation*}
   \|f - Q_{f,P,A}\|_{C^0(\Omega)}  \leq (1+ \Lambda(P_A) \|S_{A,P}\|_\infty)\|f - Q_{f,A}\|_{C^0(\Omega)} + \mu\Lambda(P_A) \|S_{A,P}\|_\infty \,,
\end{equation*}
where $\Lambda(P_A)$ denotes the Lebesgue constant of the (non-tensorial) interpolation
\linebreak
nodes $P_A\subseteq \Omega$, $S_{A,P} \in \R^{|A|\times |P|}$ with $S_{A,P}R_{A,P}  = \mathrm{Id}_{\R^{|A| \times |A|}}$ the \emph{Moore–Penrose pseudo-left-inverse}, see, e.g.,\cite{ben2003,Lloyd_Num}, and $\mu = \max_{p \in P} | f(p) - Q_{f,P,A}(p)|$ the regression error.

The \emph{approximation factor} $\Lambda(P_A) \|S_{A,P}\|_\infty$ measures both, the approximation power and the stability of the regression scheme.
Especially, for equispaced data, the Euclidean $l_2$-degree factor scales similarly  to the one of the total $l_1$-degree, and is significantly smaller than the one appearing for the maximum $l_\infty$-degree (see Section~\ref{sec:Num}). Conversely,  the Euclidean degree interpolation achieves almost the same fast exponential rate as the $l_\infty$-degree interpolation can reach, Eq.~\eqref{Rate}. In combination, the performance of Euclidean degree regression is
superior to the other regressions.
\item[iii)] \emph{Adaptive regression:} We combine $i)$,$ii)$ and present an adaptive high order domain decomposition regression scheme that delivers a global polynomial approximation $Q_{f,A}$, $A = A_{m,n,p}$ of the function $f: \Omega \lo \R$, $Q_{f,A}\approx f$ on the whole domain $\Omega$ by merging the individual polynomials of each subdomain by global multivariate interpolation \cite{MIP}. The concept substantially differs from \cite{boyd2011exponentially}, where domain decomposition for interpolation schemes is presented and analysed. We further demonstrate that the globally merged polynomials $Q_{f,A}$
reach close approximations, being non-reachable  for global (Euclidean) regression while already running into instability, see Section~\ref{sec:Num}.

We want to emphasise that the flexibility of general $l_p$-degree-merging comes by our introduced regression scheme in $(i)$, but does not apply for the prominent tensorial ($l_\infty$-degree) Chebyshev-polynomial interpolation \cite{chebfun,Lloyd}.
 \end{enumerate}
In summary, we show that when considering multivariate polynomial regression tasks, choosing the Euclidean degree yields an approximation scheme that is more powerful than choosing total or maximum degree.
There is a large number of previous works that address the subject of stability for fast function approximation resting on non-polynomial basis functions or specific geometric constructions.
We mention some of the approaches that directly relate to or inspired our work.

\subsection{Related work}
An excellent overview of existing approaches addressing approximation tasks (on equispaced data) is given by \cite{boyd2009divergence,boyd2011exponentially,platte:2011}.
For approaches overcoming Gibb's phenomenon we recommend \cite{tadmor2007filters} as an excellent survey on the subject.
Here, we give a short extracted list:

1. \emph{Interpolation:} As discussed above global interpolation is highly sensitive to Runge's phenomenon and becomes rapidly unstable for equispaced data
\cite{runge,hewitt1979gibbs,schonhage1961fehlerfortpflanzung,turetskii1940bounding}.
By a conformal change of variables, better polynomial interpolants can be constructed that resist Runge's phenomenon in the interior $[-1+\ee,1-\ee]$, $\ee >0$ and localise it near the boundary
\cite{hale2008new,hale2009use}.

The concept also applies to conformal rational-function-interpolation proposed by Baltensperger, Berrut, and No{\"e}l \cite{baltensperger1999exponential} and has been extended to
\emph{Floater-Hormann interpolation} \cite{floater}. However,
for equispaced data, stability can not be guaranteed in general.

2. \emph{Regression:} While polynomial regression reduces Runge's phenomenon, it can not be entirely eliminated \cite{boyd2009divergence,rakhmanov2007bounds}. However, for
high-resolution grids, interpolation in \emph{Mock-Chebyshev-sub-grids} completely defeats Runge's phenomenon, but comes with the cost of achieving only root exponential rates $\Oc(r^{\sqrt{n}})$, $r>1$, $n \in \N$ \cite{stengle1989chebyshev,boyd2009divergence}.

3. \emph{Regularisation techniques:} By choosing a polynomial degree $|A_{m,n,p}| \gg |P|$, additional constraints, as proposed in the (unpublished) work of Platte or in  \cite{boyd1992defeating}, can be formulated to suppress Runge's phenomenon for the regressor $Q_{f,P,A}$ fitting a function $f$.
Alternatively, \cite{berzins2007adaptive}  extend the concept of non-oscillatory (ENO) schemes
\cite{harten1987uniformly} in a similar regime.

4. \emph{Adaptive approximation:} Based on \cite{hale2008new,hale2009use}, (exponentially fast) adaptive rational functions interpolation was developed by \cite{tee2006rational}, whereas
superpolynomial approximation rates can be reached for some functions by the adaptive rational function regression scheme of \cite{wang2010rational}.

5. \emph{Radial basis functions:} Approximations schemes resting on radial basis functions were developed and investigated in \cite{driscoll2002interpolation,wendland2005scattered,platte2005polynomials,platte2011fast,fornberg2007runge} and reach high accuracy, especially for scattered data.
Strategies for stability of a vast class of approximation tasks on equispaced data were given in \cite{fornberg2004stable,fornberg2008stable}.

6. \emph{Spline interpolation:} Algebraic approximation rates are achievable by the prominent \emph{spline interpolation} methods, mainly developed by \cite{Boor:BS,Boor:tensorSP,Boor:SP,Boor:wings}, defeating Runge's
phenomenon when increasing the mesh resolution.

We want to note that most of the approaches above address approximations for univariate or bivariate functions (1D or 2D). While extensions to higher dimensions
can straightforwardly be given (by tensorial $l_\infty$-degree formulations), for most of the techniques, general $l_p$-degree schemes do either not match to the setup of the approach or were yet rarely investigated and
realised, as it is done in this article.

Finally, we want to emphasise that Euclidean degree regression does not overcome the introduced limitations (Section~\ref{sec:limit}) as some of the approaches above do, but extends the stability of the regression in Eq.~\eqref{eq:REG}, yielding
more approximation power and consequently reaches a broader class of continuous functions $T(\Omega,\R)\subseteq C^0(\Omega,\R)$.
Our results suggest to consider Euclidean degree regression as baseline for advancing some of the mentioned techniques above, especially, in the case of  higher-dimensional regression tasks $\dim =m  \geq 3$
for equispaced data.

\subsection{Notation}\label{sec:Notation}
Let $m,n \in \N$, $p\geq 1$. Throughout this article, $\Omega=[-1,1]^m$ denotes the $m$-dimensional \emph{standard hypercube} and $C^0(\Omega,\R)$ the
\emph{Banach space}
of continuous functions
\linebreak
$f : \Omega \lo \R$ with norm $\|f\|_{C^0(\Omega)} = \sup_{x \in \Omega}|f(x)|$. Further, we denote by  $e_i \in \R^m$, $i =1,\dots,m$ the standard basis, by $\|\cdot\|_p$ the $l_p$-norm on $\R^m$, and by $\|M\|_p$ the $l_p$-norm
of a matrix $M\in \R^{m\times m}$.

Further, $A_{m,n,p} \subseteq \N^m$ denotes all multi-indices $\alpha =(\alpha_1,\dots,\alpha_m)\in \N^m$ with $\|\alpha\|_p \leq n$.
We order $A_{m,n,p}$ with respect to the lexicographical order $\preceq$ on $\N^m$ starting from last entry to the $1$st, e.g.,
$(5,3,1)\preceq (1,0,3) \preceq (1,1,3)$.
We call $A$
\emph{downward closed} if and only if there is no $\beta = (b_1,\dots,b_m) \in \N^m \setminus A$
with $b_i \leq a_i$, $ \forall \,  i=1,\dots,m$ for some $\alpha = (a_1,\dots,a_m) \in A$ \cite{cohen3}.
The sets $A_{m,n,p}$ are downward closed for all $m,n\in \N$, $p\geq 1$ and induce a generalised notion of \emph{polynomial $l_p$-degree} as follows:

We consider the \emph{real polynomial ring} $\R[x_1,\dots,x_m]$  in $m$ variables and denote by $\Pi_m$ the $\R$-\emph{vector space of all real polynomials} in $m$ variables.
For $A\subseteq \N^m$, $\Pi_{A} \subseteq \Pi_m$ denotes the \emph{polynomial subspace} $\Pi_A = \mathrm{span}\{x^\alpha\}_{\alpha \in A}$ spanned by the (unless further specified) \emph{canonical} (monomial) \emph{basis}. Choosing $A =A_{m,n,p}$ as in Eq.~\eqref{eq:size} yields the spaces
$\Pi_{A_{m,n,p}}$.
Given linear ordered sets $A \subseteq \N^m$, $B \subseteq \N^n$, we slightly abuse notation by writing matrices $R_{A,B}\in \R^{|A|\times |B|}$ as
\begin{equation}
 R_{A,B} = (r_{\alpha,\beta})_{\alpha\in A, \beta \in B} \in \R^{|A|\times|B|}\,,
\end{equation}
where $r_{\alpha,\beta} \in \R$ is the $\alpha$-th, $\beta$-th entry of $R_{A,B}$.
Finally, we use the standard \emph{Landau symbols} $f \in \Oc(g)  \Longleftrightarrow \lim \sup_{x\rightarrow \infty} \frac{|f(x)|}{|g(x)|} \leq \infty$,
$f \in o(g)  \Longleftrightarrow \lim_{x\rightarrow \infty} \frac{|f(x)|}{|g(x)|} =0$.

\section{Multivariate regression for downward closed multi-indices}

While 1D polynomial interpolation and regression is a classic \cite{LIP,ConteSamuelDaniel2017Ena:}. We, here, address its challenges in multi-dimensions.
To start with, we extract the essential ingredients from \cite{cohen2,cohen3,PIP1,PIP2,MIP,IEEE} on which our approach rests.

\subsection{The notion of unisolvence}
For a downward closed multi-index set $A \subseteq \N^m$, $m\in N$, and the induced polynomial space $\Pi_A$, a set of nodes $P \subseteq \Omega$ is called \emph{unisolvent with respect to $\Pi_A$}
if and only if there exists no hypersurface $H = Q^{-1}(0)$
generated by a polynomial $0\not =Q \in \Pi_A$ with $P \subseteq H$.
In fact, unisolvence ensures the uniqueness of the interpolant $Q_{f,A}(p_\alpha) = f(p_{\alpha})$, $Q_{f,A} \in \Pi_A$. The following notion, will be a crucial later on:

\begin{definition}[1$^\text{st}$  and 2$^\text{nd}$  essential assumptions]
\label{def:EA}
Let $m \in \N$, $A\subseteq \N^m$ be a downward closed set of multi-indices, and $\Pi_A \subseteq \Pi_m$ the polynomial sub-space induced by $A$.
We consider the generating nodes given by the grid
\begin{equation}\label{GP}
\mathrm{GP}= \oplus_{i=1}^m P_i\,, \quad P_i =\{p_{0,i},\dots,p_{n_i,i}\} \subseteq \R \,, \,\,\,n_i=\max_{\alpha \in A}(\alpha_i)\,,
\end{equation}
\begin{equation}\label{eq:UN}
   P_A = \li\{ (p_{\alpha_1,1}\,, \dots \,, p_{\alpha_m,m} ) : \alpha \in A\re\} \,.
 \end{equation}
\begin{enumerate}
  \item[A1)] If the $P_i \subseteq [-1,1]$ are arbitrary distinct points then the node set $P_A$
  is said to satisfy the \emph{1$^\text{st}$ essential assumption}.
  \item[A2)] We say that the \emph{2$^\text{nd}$ essential assumption} holds if in addition the $P_i$ are chosen as the Chebyshev-Lobatto nodes that, in addition, are Leja-ordered \cite{leja}, i.e,
  \begin{equation*}
  P_i =\{p_0,\dots,p_n\} = \pm \Cheb_n = \li\{ \cos\Big(\frac{k\pi}{n}\Big) : 0 \leq k \leq n\re\}
  \end{equation*}
  and the following holds
  \begin{equation}\label{LEJA}
   |p_0| = \max_{p \in P}|p|\,, \quad \prod_{i=0}^{j-1}|p_j-p_i| = \max_{j\leq k\leq m} \prod_{i=0}^{j-1}|p_k-p_i|\,,\quad 1 \leq j \leq n\,.
  \end{equation}
  Note that in the $p=\infty$ case the same nodes result regardless of the ordering of $\Cheb_n$.
\end{enumerate}
\end{definition}
Note that $\{p_0,p_1\} = \{-1,1\}$ for all Leja-ordered Chebyshev-Lobatto nodes $\Cheb_n$ with $n\geq 1$.
Points $P_A$ that fulfill the \emph{1$^\text{st}$ essential assumption}, $(A1)$ form non-tensorial (non-symmetric) grids and are unisolvent with respect to $\Pi_A$, \cite{cohen2,cohen3,PIP1,PIP2,MIP,IEEE}.
This allows generalising classic interpolation approaches to higher dimensions.

\subsection{Newton and Lagrange interpolation}

Given unisolvent nodes $P_A$, multivariate generalisations of the classic 1D Newton and Lagrange interpolation schemes, \cite{LIP}, were given in \cite{cohen2,cohen3,PIP1,PIP2,MIP,IEEE}, which we summarize here as follows.

\begin{definition}[Lagrange polynomials]Let $m \in \N$, $A\subseteq \N^m$ be a downward closed set of multi-indices, and $\Pi_A \subseteq \Pi_m$ be a set of unisolvent nodes satisfying
  $(A1)$ from Definition~\ref{def:EA}.
We define the \emph{multivariate Lagrange polynomials}  $L_\alpha \in \Pi_A$ by
\begin{equation}
  L_\alpha(p_\beta) = \delta_{\alpha,\beta}\,,
\end{equation}
where $\delta_{\cdot,\cdot}$ is the Kronecker delta.
\end{definition}

Since the $|A|$-many Lagrange polynomials are linearly independent functions, and
\linebreak
$\dim \Pi_A =|A|$, the Lagrange polynomials are a basis of $\Pi_A$. Consequently, for any function $f : \Omega \lo \R$, the unique interpolant $Q_{f,A} \in \Pi_A$ with $Q_{f,A}(p_{\alpha}) = f(p_{\alpha})$, $\forall \alpha \in A$ is given by
\begin{equation}\label{eq:LAG}
  Q_{f,A}(x) = \sum_{\alpha \in A}f(p_{\alpha})L_{\alpha}(x)\,, \quad x \in \R^m\,.
\end{equation}
However, while the Lagrange polynomials are rather a mathematical concept, this does not assert how to evaluate the interpolant $Q_{f,A}$ at an argument $x_0 \not \in P_A \subseteq \R^m$.
For that purpose we introduce:

\begin{definition}[Newton polynomials] Let $A\subseteq \N^m$ be a downward closed set and $P_A\subseteq \R^m$ fulfill $(A1)$ from Definition~\ref{def:EA}. Then,
the \emph{multivariate Newton polynomials} are given by
 \begin{equation}\label{Newt}
  N_\alpha(x) = \prod_{i=1}^m\prod_{j=0}^{\alpha_i-1}(x_i-p_{j,i}) \,, \quad \alpha \in A\,.
 \end{equation}
 \label{def:LagP}
\end{definition}
Indeed, in dimension $m=1$ the concepts above reduce to the classic definition of Lagrange and Newton polynomials~\cite[see e.g.][]{gautschi,Stoer,Lloyd}. Moreover, also the Newton polynomials are bases of $\Pi_A$ and the computation of the interpolant in Newton form as well as its evaluation and differentiation can be realised numerically accurately and efficiently, \cite{minterpy}.
\begin{remark}
 For $A= A_{m,n,\infty}$ the grid $P_A$ becomes tensorial when fulfilling $(A1)$ from Definition~\ref{def:EA}, and the above definition recovers the known  \emph{tensorial $m$D Lagrange interpolation}:
\begin{equation}\label{LagTens}
  L_{\alpha}(x)= \prod_{i=1}^m l_{\alpha_i,i} (x) \,, \quad l_{j,i} (x) = \prod_{h=0, h \not = j}^n \frac{x_i-p_{h,i}}{p_{j,i} - p_{h,i}}  \,,  \quad  1 \leq i,j\leq n\,, \alpha \in A
\end{equation}
where $p_{j,i} \in P_i$ as in Definition~\ref{def:EA}.
Using the Newton interpolation \cite{MIP,minterpy} with $f = L_{\alpha}$, an explicit formula for the Lagrange polynomials can be derived also in the  non-tensorial case, e.g., $A=A_{m,n,2}$, yielding
\begin{equation}\label{LN}
 L_{\alpha}(x) = \sum_{\beta \in A}c_\beta N_\beta(x)\,.
\end{equation}
\end{remark}

We extend the concept of Newton-Lagrange interpolation by setting up a regression scheme in the next section.

\subsection{Newton-Lagrange regression}
\label{sec:NL_REG}

We combine the present multivariate Newton and Lagrange interpolation schemes for realising the following regression scheme.

\begin{definition}\label{def:REG} Let  $A\subseteq \N^m$ be a downward closed set and $P_A\subseteq \R^m$
  fulfill $(A1)$ from Definition~\ref{def:EA}.
  Denote with
  $\{L_{\alpha}\}_{\alpha \in A}$ the corresponding Lagrange basis of $\Pi_A$ given in Newton form as in Eq.~\eqref{LN}. Let further $P =\{p_0,\dots,p_K\}\subseteq \Omega$ be any arbitrary set of points and
  $f : \Omega \lo \R$ be a continuous function. Then we consider the regression matrix $R_{A,P} \in \R^{|P|\times|A|}$ as in Eq.~\eqref{eq:RA}
  \begin{equation*}
    R_{A,P} = \big (L_{\alpha}(p_i)\big)_{i= 1,\ldots,|P|, \alpha \in A }\,.
  \end{equation*}
Assuming that $R_{A,P}$ has full $\rank R_{A,P} = |A|$ the solution $C \in \R^{A|}$ of the least square problem
\begin{equation}
  C=\mathop{argmin}_{X \in \R^{|A|}} \|F - R_{A,P}X\|^2\,, \quad F=(f(p_i))_{i=1,\ldots,|P|} \in \R^{|P|}
\end{equation}
is uniquely determined and we denote with
  \begin{equation}
    Q_{f,P,A}  = \sum_{\alpha \in A} c_\alpha L_{\alpha}\in \Pi_A \,, \quad C=(c_{\alpha})_{\alpha \in A} \in \R^{|A|}
  \end{equation}
the \emph{regressor} fitting $f$ in $P$ with respect to the chosen degree $A\subseteq \N^m$ and nodes $P_A\subseteq \Omega$.
\end{definition}

In the next section the approximation power of the regression scheme is investigated in regard of the initially introduced perspective.

\section{Approximation power of multivariate $l_p$-degree regression}
\label{sec:App}

The present Newton-Lagrange regression scheme derives polynomials that seek to approximate general continuous functions. Here, we
bound the approximation errors and theoretically argue when uniform convergence for regular functions can be guaranteed.

\subsection{Lebesgue constants and approximation errors}
The Lebesgue constant is a crucial key for estimating approximation errors of polynomial interpolants of general continuous functions.
We start by defining:

\begin{definition}[Lebesgue constant] \label{def:LEB} Let $m \in \N$, $A\subseteq \N^m$ be a downward closed set of multi-indices, $P_A\subseteq \Omega$ be a set of unisolvent nodes satisfying $(A1)$ from Definition~\ref{def:EA}.
Let  $f \in C^0(\Omega,\R)$ and $Q_{f,A}(x) = \sum_{\alpha \in A}f(p_\alpha)L_{\alpha}(x)$ be its Lagrange interpolant.
Then, we define the \emph{Lebesgue constant} analogously to the 1D case, see e.g. \cite{gautschi}, as
\begin{align*}
\Lambda(P_A) := \sup_{f\in C^0(\Omega,R)\,, \|f\|_{C^0(\Omega)}\leq 1} \|Q_{f,A}\|_{C^0(\Omega)}
             =  \Big\|\sum_{\alpha \in A} |L_{\alpha}|\Big\|_{C^0(\Omega)}\,.
\end{align*}
\end{definition}
Based on the 1D estimate
\begin{equation}\label{eq:LEB}
 \Lambda(\Cheb_n)=\frac{2}{\pi}\big(\log(n+1) + \gamma +  \log(8/\pi)\big) + \Oc(1/n^2)\,,
\end{equation}
known for Chebyshev-Lobatto nodes, surveyed by \cite{brutman2}, \cite{cohen2,cohen3,MIP} further detail and study this concept in $m$D and show
that unisolvent nodes satisfying $(A2)$ from Definition~\ref{def:EA} induce high approximation power reflected in the small corresponding  Lebesgue constants.

In the special case of multi-dimensional $l_\infty$-Chebyshev grids, the estimate extends to multi-dimension:
\begin{lemma}\label{lemma:log}
Let $A=A_{m,n,\infty}$, $m,n\in \N$, and $P_A$ be a full tensorial grid fulfilling $(A2)$ from Definition~\ref{def:EA}.
Then
\begin{equation*}\label{LE2}
  \Lambda(P_A) \leq \prod_{i=1}^m \Lambda(\Cheb_n) \in \Oc(\log(n+1)^m)\,.
\end{equation*}
\end{lemma}
\begin{proof}We use the tensorial Lagrange polynomials $L_\alpha(x) = \prod_{i=1}^ml_{\alpha_i,i}(x_i)$ with $l_{j,i}$ given in Eq.~\eqref{LagTens}.
This allows us to bound
\begin{align}
\Lambda(P_{A_{m,n,\infty}}) &=  \sup_{f \in C^0(\Omega,\R)}\frac{\|Q_{f,A}\|_{C^0(\Omega)}}{\|f\|_{C^0(\Omega)}}
 \leq \big\|\sum_{\alpha \in A} |L_{\alpha}| \big\|_{C^0(\Omega)}   \nonumber  \\
  &\leq \big\| \sum_{\alpha \in A} \prod_{i=1}^m |l_{\alpha_i,i}| \big \|_{C^0(\Omega)}\\
  &=\big \| \big(\sum_{j=0}^n |l_{j,1}|\big) \cdots   \big(\sum_{j=0}^n |l_{j,l}|\big)\  \cdots \big(\sum_{j=0}^n |l_{j,m}|\big)\big\|_{C^0(\Omega)}\,, \,\,\, 1<l<m\nonumber \\
 &\leq \prod_{i=1}^m\big \| \sum_{j=0}^n |l_{j,i}| \big\|_{C^0(\Omega)}=  \prod_{i=1}^m \Lambda(P_i)\,. \nonumber
\end{align}
Due to Eq.~\eqref{eq:LEB}, this yields $\Lambda(P_A) \leq \Lambda(\Cheb_n)^m \in \Oc(\log(n+1)^m)$.
\end{proof}

\begin{remark}\label{rem:LEB} In the case of general $l_p$-degree, due to \cite{cohen3} the Lebesgue constants increase, i.e,
  \begin{equation}\label{eq:LI}
    \Lambda(P_{A_{m,n,p}}) \geq  \Lambda(P_{A_{m,n,\infty}})\,, \quad p > 0
  \end{equation}
  reflecting the less-constrained Lagrange polynomials, $L_{\alpha}(p_\beta) =0$ for all $\beta \in A_{m,n,p}\subsetneq A_{m,n,\infty}$, see
  Fig.~\ref{fig:L}.
  Thus, the approximation power of
  $l_p$-degree interpolation might be less than $l_\infty$-interpolation. In \cite{Lloyd2,MIP}, however, the converse is demonstrated for \emph{Trefethen functions}, by reaching the same approximation rates with Euclidean $l_2$-degree interpolation as when using $l_\infty$-degree interpolation. Thus, the moderate increase of the Lebesgue constant,  see Fig.~\ref{fig:L}, seems not to be a limiting factor for Euclidean $l_2$-degree interpolation for (regular) Trefethen functions.
\end{remark}
Motivated by the classic Lebesgue inequality, see e.g. \cite{brutman2}, we estimate the approximation quality of the regression schemes proposed here.

\begin{theorem}\label{theo:APP} Let $A\subseteq \N^m$, $m\in\N$ be a downward closed set and $P_A\subseteq \R^m$ fulfill $(A1)$ from Definition~\ref{def:EA}.
  Let further $P =\{p_0,\dots,p_K\}\subseteq \Omega$, $K \geq |A|$ be any arbitrary set of data points and let
  $f : \Omega \lo \R$ be a continuous function.

    Assume that the regression matrix $R_{A,P}$ from Eq.~\eqref{eq:RA} has full $\rank\, R_{A,P} = |A|\leq |P|$ and let $Q_{f,P,A} \in \Pi_A$ be the regressor of $f$ in $P$ according to Definition~\ref{def:REG} with regression error
    $$ \mu = \|\widetilde F - F \|_\infty \,, \quad \quad \widetilde F = (Q_{f,P,A}(p_i))_{i=1,\ldots,K}\,, \,F = (f(p_i))_{i=1,\ldots,K} \in \R^K\,.$$
    Denote with $Q_{f,A} = \sum_{\alpha \in A}f(p_\alpha)L_\alpha \in \Pi_A$ the Lagrange interpolant of $f$ in $P_A$ from Eq.~\eqref{eq:LAG}.
    Then:
\begin{align}
   \|f - Q_{f,P,A}\|_{C^0(\Omega)} & \leq (1+ \Lambda(P_A) \|S_{A,P}\|_\infty)\|f - Q_{f,A}\|_{C^0(\Omega)} + \mu\Lambda(P_A) \|S_{A,P}\|_\infty \,, \label{eq:EST}
\end{align}
where $S_{A,P} \in \R^{|A|\times |P|}$ with $S_{A,P}R_{A,P}  = \mathrm{Id}_{\R^{|A| \times |A|}}$
is the \emph{Moore–Penrose pseudo-left-inverse}, see e.g., \cite{ben2003,Lloyd_Num}.
\end{theorem}
\begin{proof} While the nodes $P_A$ are unisolvent with respect to $\Pi_A$ the interpolation operator
$$ I_{P_A} : C^0(\Omega,\R) \lo \Pi_A \,, \quad f \mapsto Q_{f,A}$$
is a linear operator with operator norm
$$\|I_{P_A}\| = \sup_{f\in C^0(\Omega,R)\,, \|f\|_{C^0(\Omega)}\leq 1} \|Q_{f,A}\|_{C^0(\Omega)} = \Lambda(P_A)$$
given by the Lebesgue constant
from Definition~\ref{def:LEB}. In particular, $I_{P_A}(Q) = Q$ holds for all polynomials $Q \in \Pi_A$.
Denote with $\mathbb{Q}_R = (Q_{f,A,P}(p_\alpha))_{\alpha \in A}\in \R^{|A|}$ the values of the regressor in the interpolation nodes $P_A$ and with
$\mathbb{Q}_I = (Q_{f,A}(p))_{p \in P}$ the values of the interpolant in the data points $P$, and with $F = (f(p_\alpha))_{\alpha \in A} \in \R^{|A|}$ the function values in the interpolation nodes. Then we observe that $\mathbb{Q}_R = S_{A,P}\widetilde F$ and
$S_{A,P}\mathbb{Q}_I =F$ and use these identities to deduce:
  \begin{align*}
     \|f- Q_{f,P,A}\|_{C^0(\Omega)} & \leq \|f - Q_{f,A}\|_{C^0(\Omega)} + \|Q_{f,A}-Q_{f,P,A}\|_{C^0(\Omega)} \\
     &\leq \|f - Q_{f,A}\|_{C^0(\Omega)} + \|I_{P_A}(Q_{f,A}-Q_{f,P,A})\|_{C^0(\Omega)} \\
      &\leq \|f - Q_{f,A}\|_{C^0(\Omega)} + \Lambda(P_A)\|S_{A,P}(\mathbb{Q}_I-\widetilde F)\|_{\infty} \\
     & \leq \|f - Q_{f,A}\|_{C^0(\Omega)} + \Lambda(P_A) \|S_{A,P}\|_\infty\big (\|F-\mathbb{Q}_I\|_{\infty} + \|\widetilde F - F\|_\infty \big) \  \\
     & \leq (1+ \Lambda(P_A) \|S_{A,P}\|_\infty)\|f - Q_{f,A}\|_{C^0(\Omega)} + \mu\Lambda(P_A) \|S_{A,P}\|_\infty \,,
  \end{align*}
  where we used $\|F -\mathbb{Q}_I\|_{\infty} \leq \|f - Q_{f,A}\|_{C^0(\Omega)}$
  for the last estimate.
\end{proof}

Theorem~\ref{theo:APP} gives rise to the following interpretation:
\begin{remark}
  While $\Lambda(P_A)\|S_{A,P}\|_\infty$ is independent of $f$ and the regression error $\mu \in \R^+$ can be measured in practice, reasonable regression errors $\mu \in \R^+$ are given whenever the 2$^\text{nd}$ term
  $ \mu\Lambda(P_A) \|S_{A,P}\|_\infty$
  is sufficiently small. Given that, the approximation quality of $Q_{f,P,A}$
  relies on the approximation error $\|f - Q_{f,A}\|_{C^0(\Omega)}$
  of the interpolant $Q_{f,A}$ of $f$ in $P_A$ in combination with the factor $\Lambda(P_A) \|S_{A,P}\|_\infty$.

 Thus, though due to Remark~\ref{rem:LEB}  the Lebesgue constant of Euclidean  degree is higher than the one given by maximum degree, i.e., $\Lambda(P_{A_{m,n,p}}) \geq  \Lambda(P_{A_{m,n,\infty}})$, $0<p<\infty$,
 the smaller norm of the pseudo-inverse $\|S_{A,_{,m,n,2},P}\|_\infty \leq \|S_{A,_{,m,n,\infty},P}\|_\infty$ causes the converse estimate for the  \emph{approximation factor}
\begin{align}
  \Lambda(P_{A_{m,n,2}}) \|S_{A,_{,m,n,2},P}\|_\infty \leq \Lambda(P_{A_{m,n,\infty}}) \|S_{A,_{,m,n,\infty},P}\|_\infty\,, \label{eq:Factor}
\end{align}
in case $P_A$ fulfils $(A2)$ from Definition~\ref{def:EA}, see Experiment,~\ref{exp:FAC} and Fig.~\ref{fig:PInv}.
\end{remark}

We summarize the insights above:
\begin{corollary}\label{cor:uni1} Let $A\subseteq \N^m$,  $m\in\N$ be a downward closed set and $P_A\subseteq \R^m$ fulfill  $(A2)$ from Definition~\ref{def:EA}.
  Given a sequence $P_n =\{p_0,\dots,p_{K_n}\}\subseteq \Omega$, $K_n \geq |A|$ of (not necessarily nested) point sets and
  a continuous function $f : \Omega \lo \R$ with fast approximation rate given by interpolation in $P_A$
  \begin{equation}
    \|f - Q_{f,A_{m,n,p}}\|_{C^0(\Omega)}  = o\big(1 +\Lambda(P_{A_{m,n,p}})\|S_{A_{m,n,p},P_n}\|_\infty\big) \,.
  \end{equation}
  Assume that the polynomial $Q_{f,P_n,A_{m,n,p}}  \in \Pi_{m,n,p}$ fitting $f$ in $P_n$ possesses fast decreasing regression errors
  $\mu_n = \|Q_{f,P_n,A_{m,n,p}}(P_n) - f(P_n)\| = o(\Lambda(P_{A_{m,n,p}})\|S_{A_{m,n,p},P_n}\|_\infty)$ then
  \begin{equation}
    Q_{f,P_n,A_{m,n,p}} \xrightarrow[n\rightarrow \infty]{} f \,, \quad \text{uniformly on}\,\,\, \Omega\,.
  \end{equation}
\end{corollary}
\begin{proof} The proof follows directly by Theorem~\ref{theo:APP}.
\end{proof}

By summarising this section we conclude: If $f$ can be approximated due to interpolation
with an exponential rate $\|f - Q_{f,A_{m,n,2}}\|_{C^0(\Omega)} = \Oc(\rho^{-n})$  (is a Trefethen function, \cite{MIP}), with
$$\rho^{-n} \in o\big(\Lambda(P_{A_{m,n,p}})\|S_{A_{m,n,p},P_n}\|_{\infty}\big)\,,$$
then Corollary~\ref{cor:uni1} applies for Euclidean degree regression and its superiority over total or maximum degree choices becomes evident, as it is also observable by our demonstrations in Section~\ref{sec:Num}.

In order to extend the limitations of regression schemes even further, we incorporate our insights of approximation theory into the following presentation of  an
adaptive decomposition approach.

\begin{figure}[t!]
\center
\includegraphics[scale=0.18]{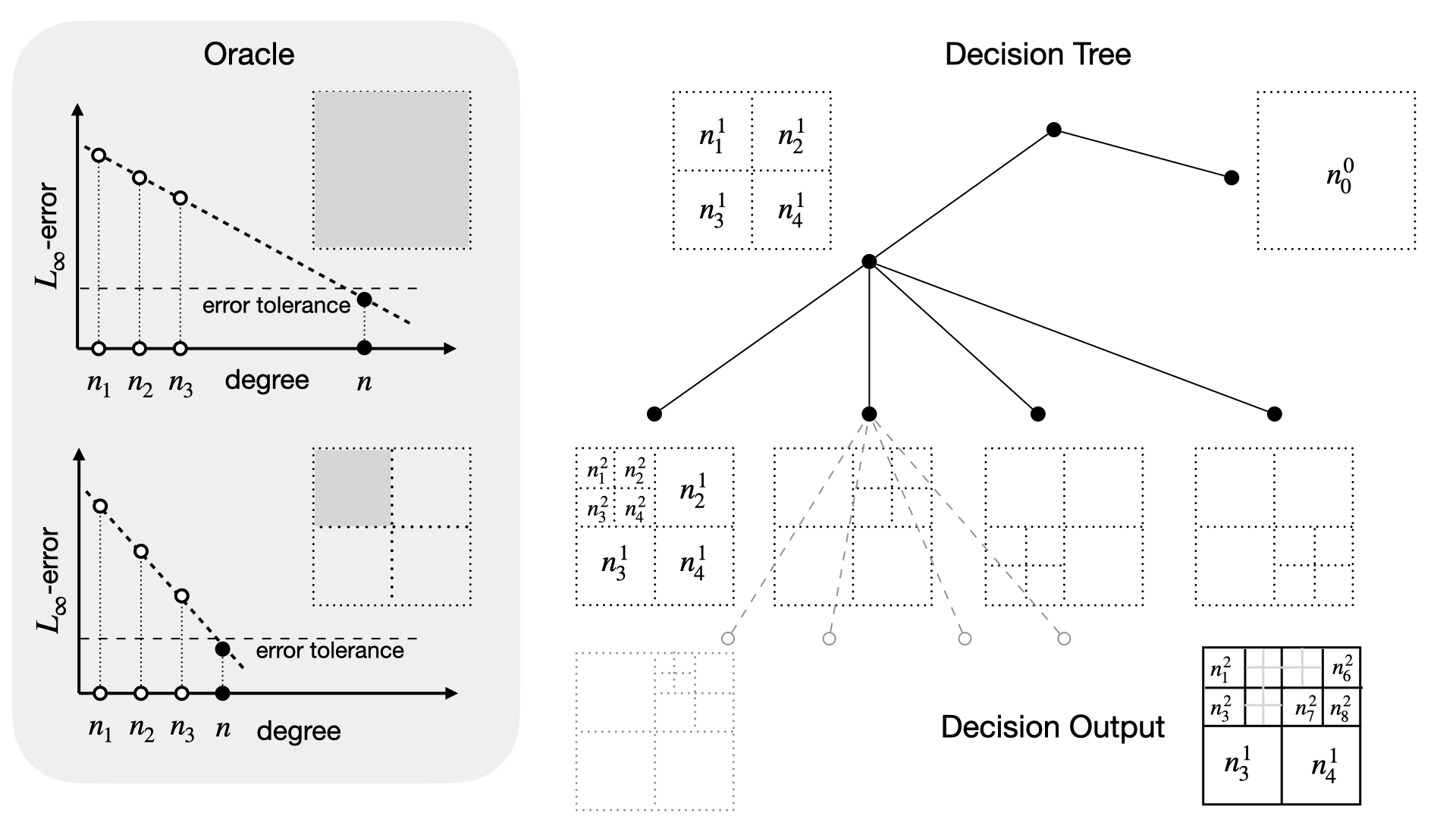}
\caption{Illustration of the oracle based adaptive domain and degree decomposition. \label{Fig:oracle}}
\end{figure}

\section{Adaptive regression of regular functions and globally merged polynomials}
\label{sec:AD}

The famous works of \cite{Boor:BS,Boor:tensorSP,Boor:SP, boorapp} address spline interpolation tasks, and rely on the fact that any sufficient regular (differentiable) function
$f: \Omega \lo R$ can be approximated by piecewise polynomial functions.

\begin{theorem}[De Boor \& H\"{o}llig] Let $f : \Omega \subseteq \R^2 \lo \R$ be a \mbox{$(n+1)$}-times continuously differentiable  function, $\Delta$ be a triangulation of $\Omega$, and
$$S_{n,\Delta} = \li\{g \in C^r(\Omega,\R) :  g_{|\delta} \in \Pi_{2,n,1}\,,  \forall\, \delta \in \Delta \re\}\,, \quad n >3r +1$$
the space of \emph{piecewise polynomial functions} of degree $n$.
Then there exists $c(\Delta) >0$ such that
 \begin{equation} \label{RC}
  \mathrm{dist}(f , S_{n,\Delta}) \leq c(\Delta) \| D^{n+1}f\|_{C^0(\Omega)}|\Delta|^{n+1}\,,
\end{equation}
where $|\Delta| = \sup_{\delta \in \Delta}\mathrm{diam}(\delta)\in \R^+$ denotes the mesh size (i.e., the maximal diameter) of the triangulation and $D^{n+1}f$ the total $(n+1)$-st derivative of the function $f$.
\end{theorem}
Generalisations of this result to higher dimensions are formulated in a similar manner~\cite{Boor:BS,Boor:tensorSP,Boor:SP}.
While every polynomial is a Trefethen function, the statement implies that any regular function can be approximated by piecewise continuous Trefethen functions.
Motivated by this fact, we deliver an \emph{adaptive, high-order, spline-type domain decomposition algorithm} relying on the results in Section~\ref{sec:TREF}. To do so, we propose the following \emph{oracle} in order to efficiently identify an \emph{(pseudo) optimal balancing} between domain decomposition and regression-degree, illustrated in Fig.~\ref{Fig:oracle}.

\algblock{Input}{EndInput}
\algnotext{EndInput}

\algblock{Output}{EndOutput}
\algnotext{EndOutput}

\begin{algorithm}[t]
\caption{Oracle based adaptive regression}\label{alg:oracle}
\begin{algorithmic}
  \Input \quad function $f:\Omega\lo \R$, error tolerance $\ee> 0$, depth $D \in \N$
  \EndInput
  \State $(\Omega_0^0, n_0^0,K^0)  \gets (\Omega, \widetilde n,1)$ \Comment{initialise input domain with predicted degree $n $  due to Eq,~\eqref{eq:predeg}}
\For{$k=0$ to $D$}
\State $b \gets 0$
\For{$j=0$ to $K^d$}
\If{$\Omega^{k-1}_j \neq \emptyset$}
\State $\{\Omega^{k}_j\}\gets \mathrm{half}(\Omega^{k-1}_j)$ \Comment{creating all subdomains by halving $\Omega_j^{k-1}$}
\State $\{n^{k}_j\}_{1\leq j \leq 2^m} \gets \mathrm{argmin}_{l \in \N} \{q_{j,\mu}(l) < \ee \}$ \Comment{predicting degree on all subdomains}
\State $\mathrm{oracle} \gets \mathrm{oracle}(f,\Omega_j^{k-1},\{\Omega_j^{k}\}_{1\leq j \leq 2^m},\ee)$ \Comment{ask oracle for  decision (Eq.~\eqref{eq:oracle})}
\If {$\mathrm{oracle} = 1$}
\State $(\Omega_j^{k-1},n_j^{k-1},K^{k-1}) \gets (\Omega_{j}^k, n_{j}^k, K^{k-1} + 2^m)$
\State $b \gets 1$
\Else
\State $(\Omega_j^{k})  \gets (\emptyset,k)$ \Comment{store decomposition depth $k$}
\EndIf
\EndIf
\EndFor
\If {$b =0$ {\bf and} error tolerance $\ee$ is achieved}
\State break \Comment{stop decomposition}
\EndIf
\EndFor
\Output \quad $(\Omega_j^{k},n_j^{k})$
\EndOutput
\end{algorithmic}
\end{algorithm}

\begin{definition}[Oracle and decision tree] Given a regular function $f : \Omega\subseteq \R^m \lo \R$, $m \in \N$ sampled on a dataset $\emptyset \neq P\subseteq \Omega$, a decomposition of $\Omega = \Omega_1 \cup \dots \cup \Omega_{2^m}$ by halving into sub-hypercubes, and let $P_{A_j}$, $A_{m,n_j,p}$
be the unisolvent nodes rescaled to the subdomains $\Omega_j$, $j=1,\dots,2^m$. By applying Newton-Lagrange regression (Definition~\ref{def:REG}), we compute the regressor  $Q_{f,A_{m,k,p},P}^0$ on $\Omega$, and the regressors  $Q_{f,A_{m,k,p},P}^j$ fitting $f$ in $P\cap \Omega_j$ and evaluate their regression errors for a sequence of degree choices $k\in \N$.
Given an error tolerance $\ee >0$, we fit the regression errors $\mu_{j,k}$ accordingly to the model $q_\mu (k) \approx c\rho^{-k}$ (with $c,\rho$ unknown) and predict the minimum polynomial degree
\begin{equation}\label{eq:predeg}
n_j  = \mathrm{argmin}_{k \in \N} \{q_\mu(k) <\ee\}
\end{equation}
reaching the required accuracy.
The option with the simpler polynomial model is defined as the oracle choice
\begin{equation}\label{eq:oracle}
  \mathrm{oracle} = \li\{\begin{array}{cc}  0 & \,, \quad \text{if} \quad |A_{m,n_0,p}| \leq \sum_{j=1}^{2^m} |A_{m,n_j,p}| \\
                                            1 & \,, \quad \text{if} \quad |A_{m,n_0,p}| > \sum_{j=1}^{2^m} |A_{m,n_j,p}|
  \end{array}\re.\,.
\end{equation}
As formalised in Algorithm~\ref{alg:oracle}, recursion  reaches an (pseudo) optimal decomposition choice, whereas the required error tolerance $\ee$ is verified for each subdomain and the decomposition is repeated in the case of violation.
\label{def:orac}
\end{definition}
Once given the piecewise  polynomials fitting the function $f : \Omega \lo \R$,  we propose to apply Lagrange interpolation in order to derive the following global polynomial approximation.
\begin{definition}[Globally merged polynomial]
\label{def:MERGE}
Let the assumptions of Definition~\ref{def:orac} be satisfied and denote with $Q_{f,A_{m,k,p},P}^j$ the piecewise polynomial fits of $f$ on the subdomains $\Omega_j^{k}$. Given global
unisolvent nodes  $P_{A_{m,n,p}}\subseteq \Omega$ satisfying $(A2)$ from Definition~\ref{def:EA}, we observe
$P_{A_{m,n,p}} \cap \partial\Omega_j^{k} = \emptyset$, $\forall j,k \in \N$ and
 define the following \emph{globally merged polynomial}
\begin{equation}\label{eq:merge}
    Q_{\mathrm{merge},A_{m,n,p}} = \sum_{\alpha \in A_{m,n,p}} Q_{f,A_{m,k,p},P}^j(p_\alpha)L_\alpha\,, \quad \text{with} \,\,\, p_{\alpha} \in  \Omega_j^{k}
\end{equation}
due to Lagrange interpolation, Eq.~\eqref{eq:LAG}.
\end{definition}
Note that, similar to \emph{spline interpolation}, the piecewise polynomial regressors coincide in the overlapping boundaries $\partial\Omega_j^{k}$ up to the pre-defined error tolerance, while the globally merged polynomial is certainly $C^\infty$-smooth.

As introduced in Section~\ref{sec:limits}, apart from storage or runtime costs, regression schemes are limited by
their stability. In contrast, interpolation in the proposed unisolvent nodes $P_A$ applies for a larger range of approximation tasks \cite{MIP}, extending the  practical limits of multivariate polynomial regression when applied as proposed above.
We evaluate this concept as part of the numerical experiments given in the next section.

 \begin{figure}[t!]
\center
\includegraphics[width=0.8\textwidth]{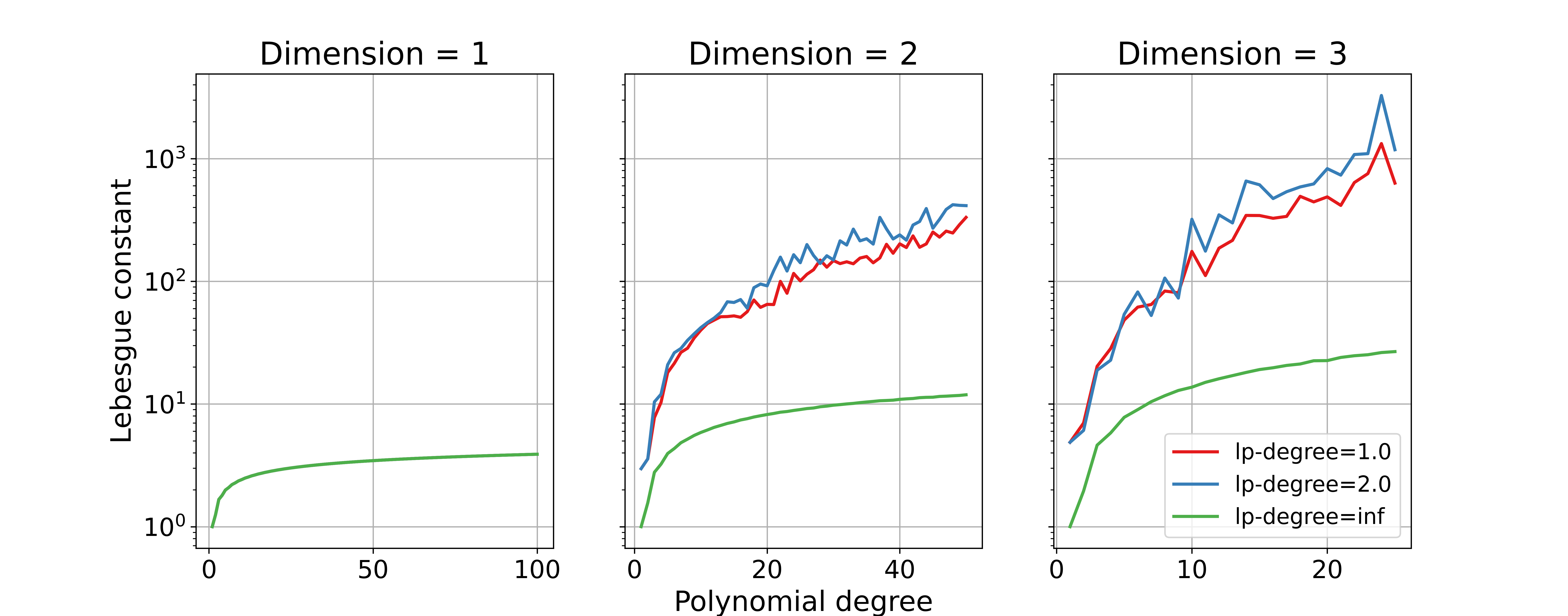}
\caption{Lebesgue constants of the proposed unisolvent nodes, with respect to the total, Euclidean, and maximum $l_p$-degree, fulfilling
  $(A2)$ from Definition~\ref{def:EA} in dimensions $m =1$ (left), $m=2$ (middle), $m=3$ (right). \label{fig:L}}
\end{figure}

\section{Numerical Experiments}\label{sec:Num}

We implemented the present Newton-Lagrange regression approach, Section~\ref{sec:NL_REG},
in {\sc Python} as part of the package {\sc minterpy} \cite{minterpy}.
The following experiments are designed for validating our theoretical findings, Theorem~\ref{theo:APP}, and proposed strategies for extending the capability of fast function approximations. By crosschecking with the classic \emph{Chebyshev regression scheme} \cite{chebfun,Lloyd} in Experiment~\ref{exp:REG}, we obtain indistinguishable results, see also Fig.~\ref{fig:FIT1}. This  validates our approach to be well-posed and realised. Consequently, we focus on the article's scope, investigating the performance for different $l_p$-degree choices and refer to \cite{chebfun,Lloyd} for further comparison with other regression schemes.

All numerical experiments were run on a standard Linux laptop (Intel(R) Core(TM) i7-1065G7 CPU @1.30GHz, 32\,GB RAM), using the standard {\sc scipy} least square solver {\em lstsq} (version 1.9.1)  for executing the regression tasks within seconds.

\begin{experiment}[Lebesgue constant] We compare the Lebesgue constant $\Lambda(P_{A_{m,n,p}})$, Definition~\ref{def:LEB} for the unisolvent nodes fulfilling $(A2)$ from Definition~\ref{def:EA} with respect to $l_p$-degrees $p =1,2,\infty$ in dimensions $m =1,2,3$. To do so, we sample $10.000$ random points $P \subseteq \Omega$ from $\Omega=[-1,1]^m$ and measure
$$ \max_{x \in P}\sum_{\alpha\in A_{m,n,p}}|L_\alpha(x)|$$
for each setting of $n \in \N$, $p=1,2,\infty$ in the same points $P$ (for each dimension).
\label{exp:LEB}
\end{experiment}

 \begin{figure}[t!]
\center
\begin{tabular}{c}
\includegraphics[width=0.95\textwidth]{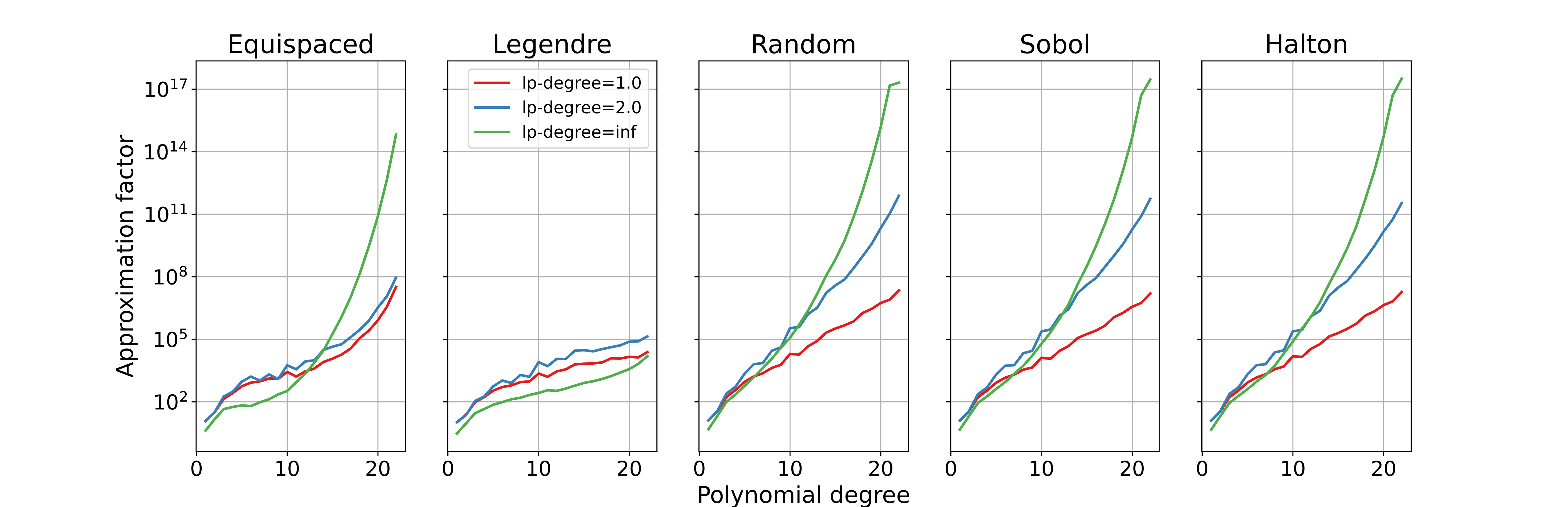}\\
\includegraphics[width=0.95\textwidth]{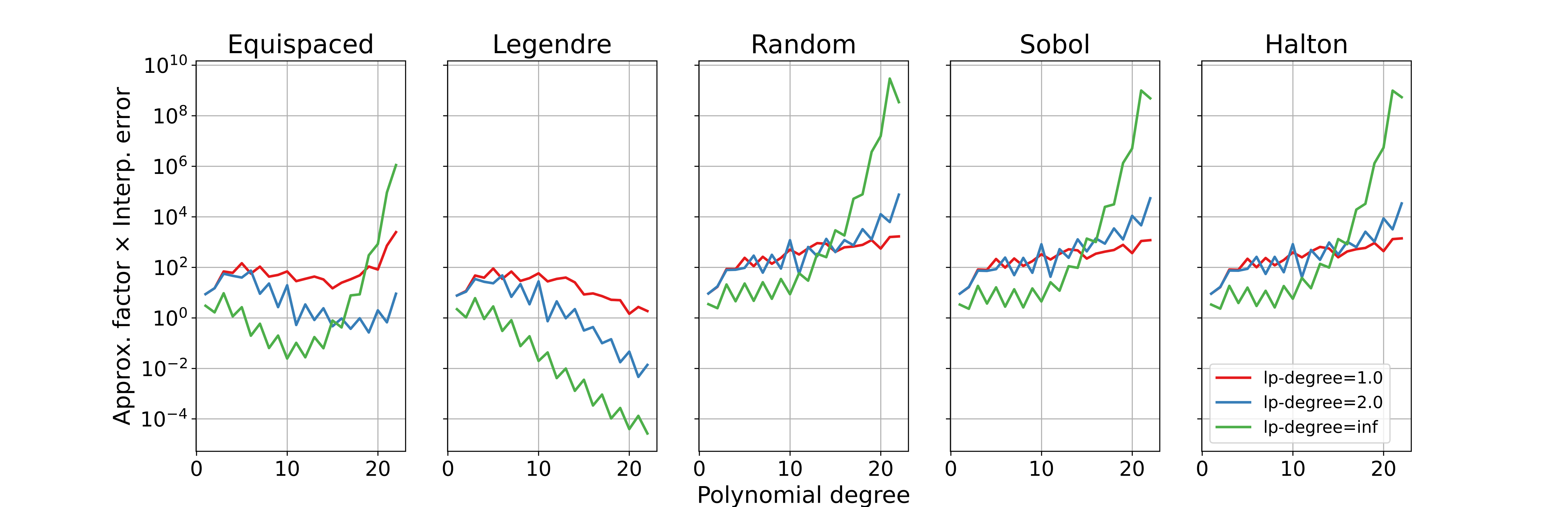}
\end{tabular}
\caption{Approximation factor for several data distributions with respect to the total, Euclidean  and  maximum $l_p$-degree in dimension  $m=3$ (first row), multiplied with the interpolation error (second row).
\label{fig:PInv}}
\end{figure}

The results are reported in Fig.~\ref{fig:L}. Note that, in dimension $m=1$ all multi-index sets coincide, i.e., $A_{1,n, p} =A_{1,n,p'}$, $p,p'>0$, resulting in only one reported Lebesgue constant for the Cheybyshev-Lobatto nodes $\Cheb_n$. We observe that the logarithmic behaviour, Eq.~\eqref{eq:LEB}, for the $l_\infty$-degree is maintained when increasing the dimension $m =2,3$,
as predicted by Lemma~\ref{lemma:log}. The Lebesgue constants for the total $l_1$-degree and the Euclidean  $l_2$-degree scale compatibly, and show a algebraic or sub-exponential growth with increasing degree $n \in \N$, as being theoretically proven in \cite{cohen3}. We observe that for degree $n=30$ in dimension $m=3$ the increase of the Lebesgue constants, compared to $l_\infty$-degree, is moderately bounded by to 2 orders of magnitude.

\begin{figure}[t!]
\center
\begin{tabular}{c}
\includegraphics[width=0.95\textwidth]{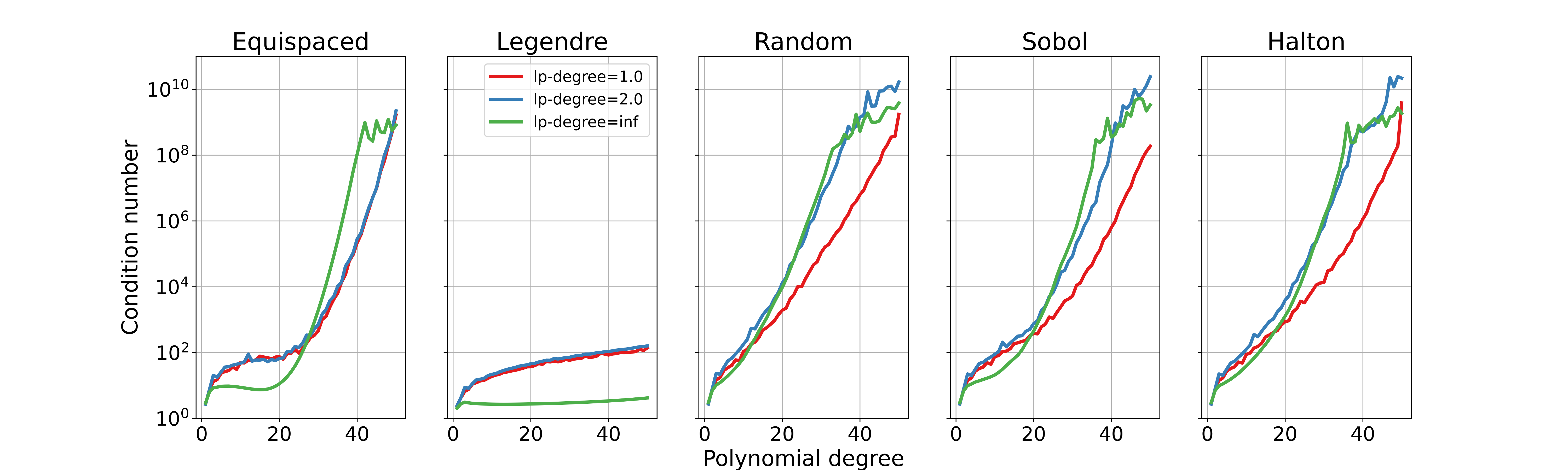}\\
\includegraphics[width=0.95\textwidth]{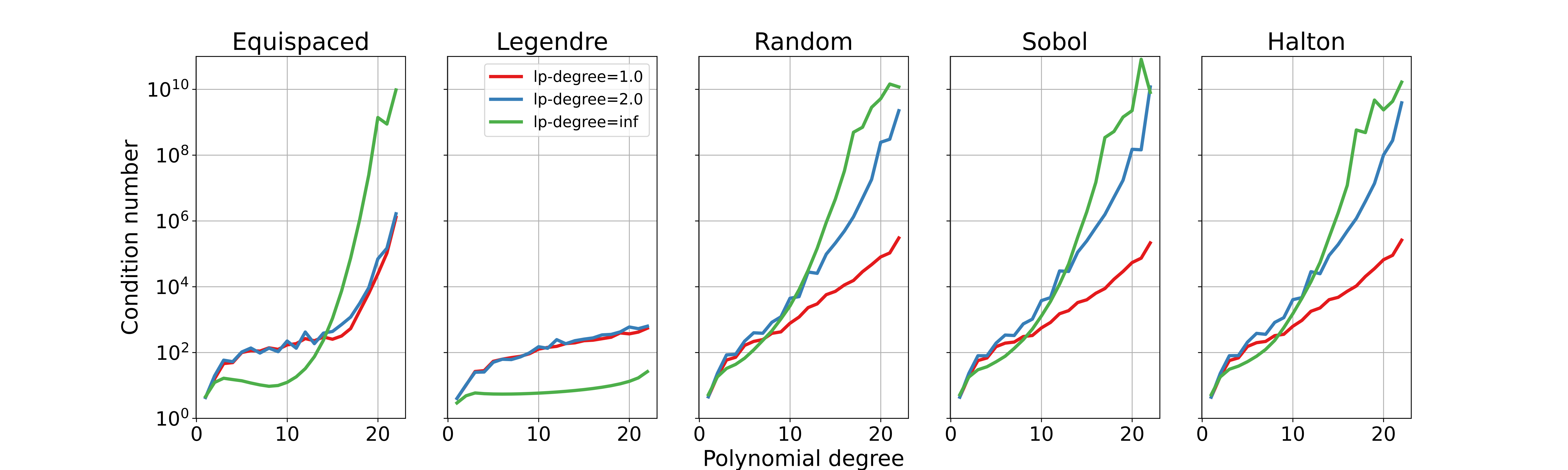}
\end{tabular}
\caption{Condition numbers of the regression matrices $R_{A,P}$, $A= A_{m,n,p}$ for several data distributions $P$ with respect to the total, Euclidean  and  maximum $l_p$-degree, $p =1,2,\infty$, in dimensions $m =2$ (first row), $m=3$ (second row). \label{fig:Cond}}
\end{figure}

\begin{figure}[ht!]
\center
\vspace{-5pt}
\begin{tabular}{cc}
\includegraphics[width=0.44\textwidth]{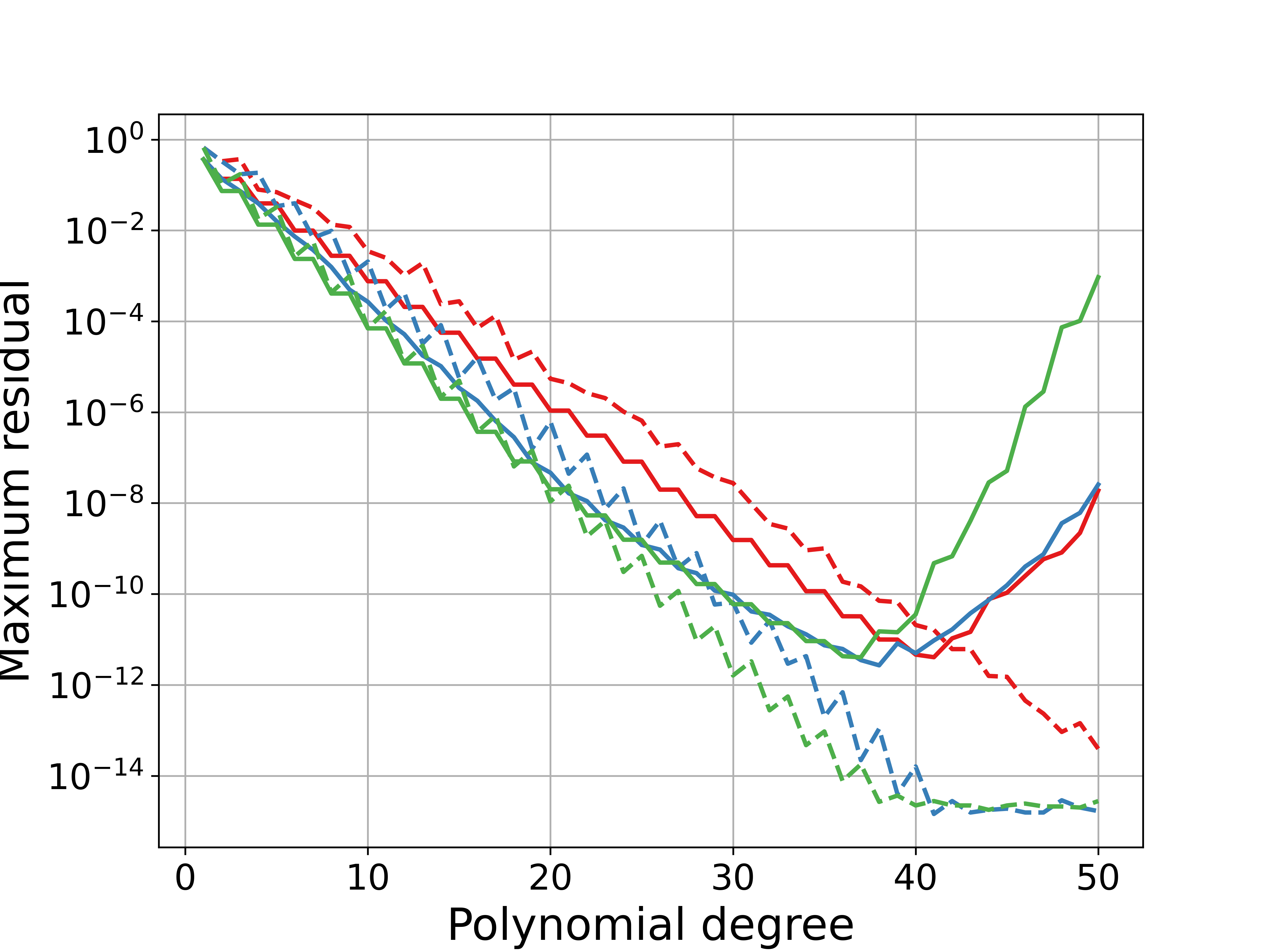}
& \includegraphics[width=0.44\textwidth]{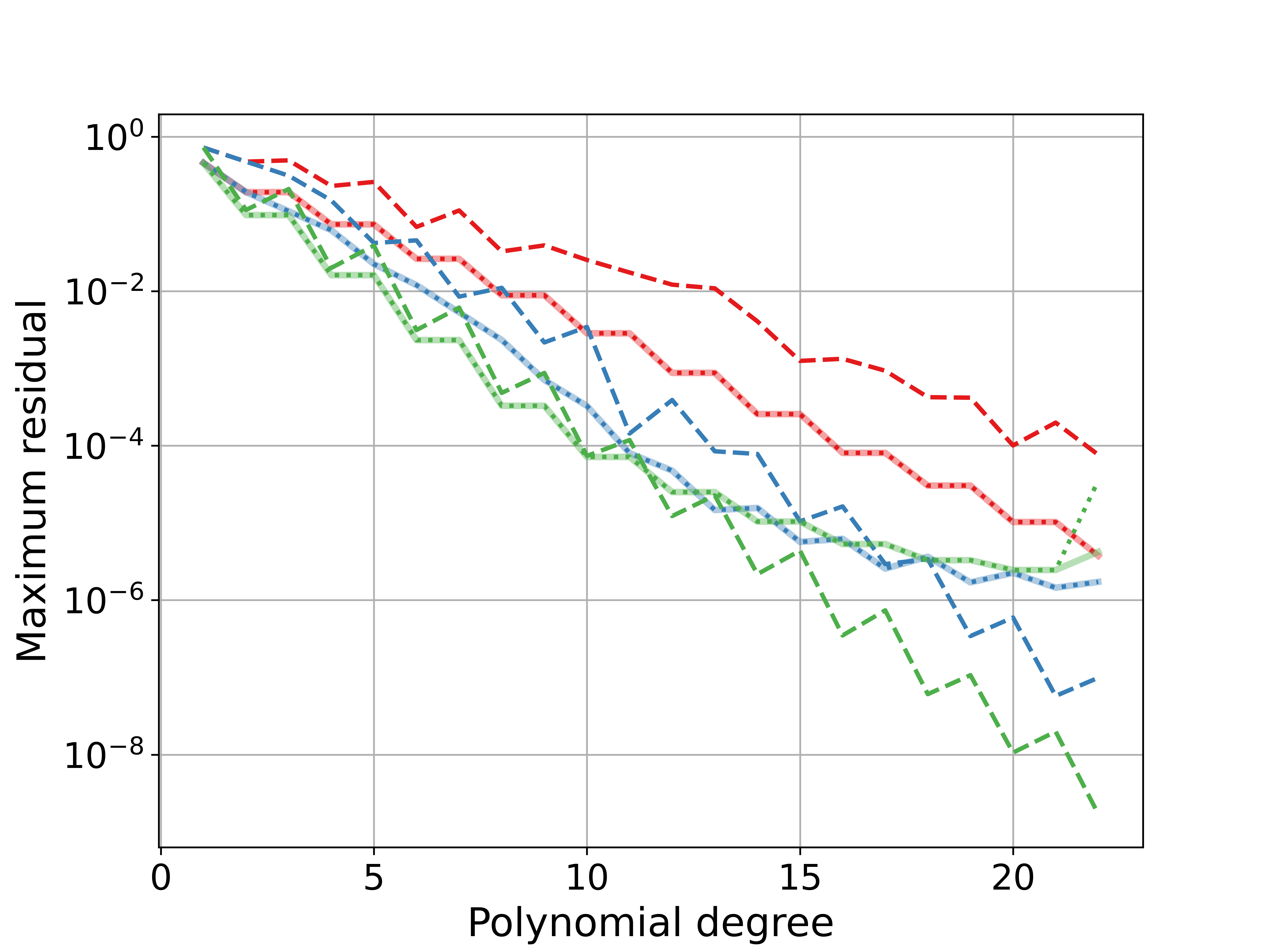}\\
\footnotesize{Equispaced 2D} & \footnotesize{Equispaced 3D}\\
\vspace{-3pt}
\includegraphics[width=0.44\textwidth]{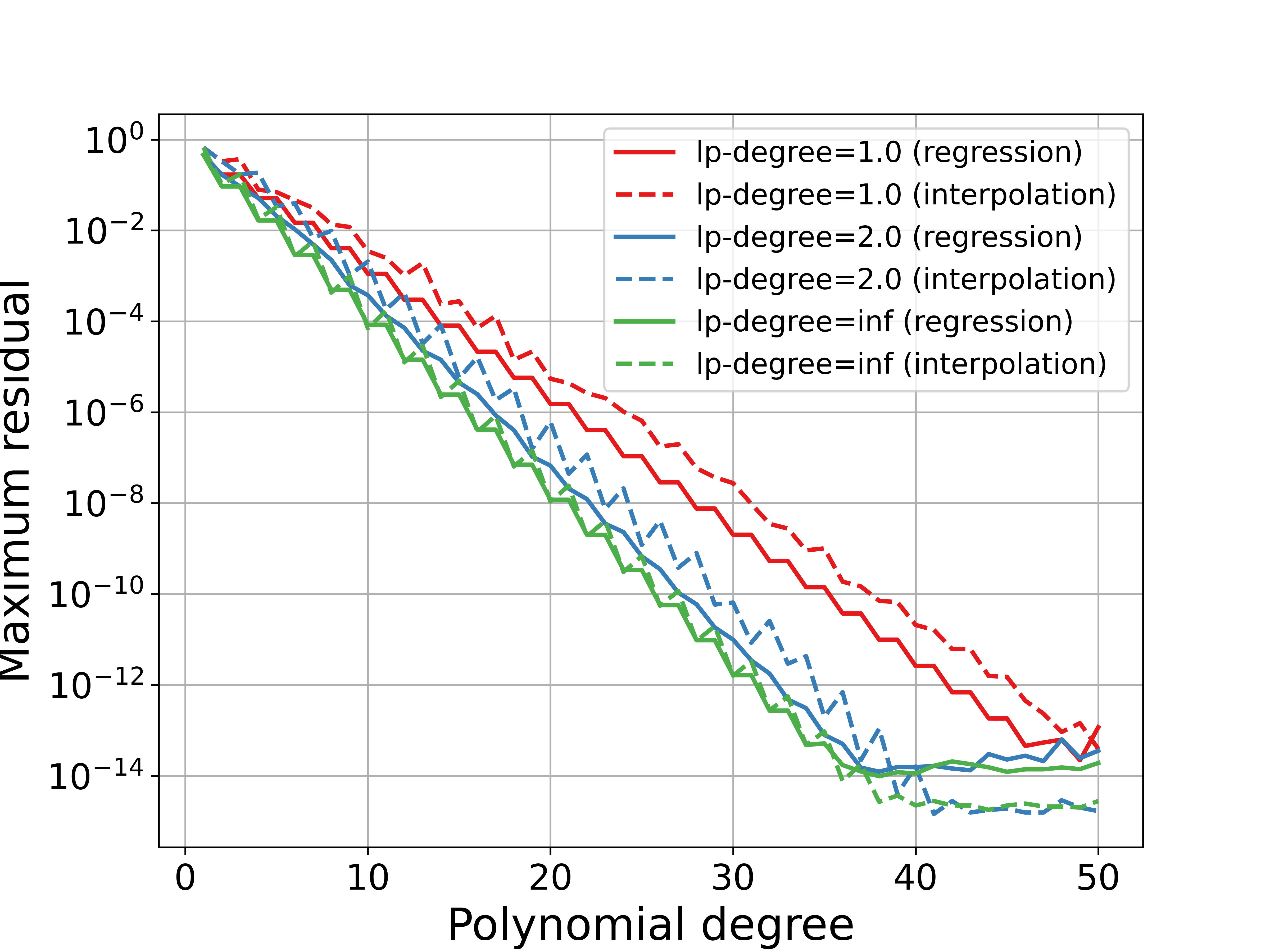}
& \includegraphics[width=0.44\textwidth]{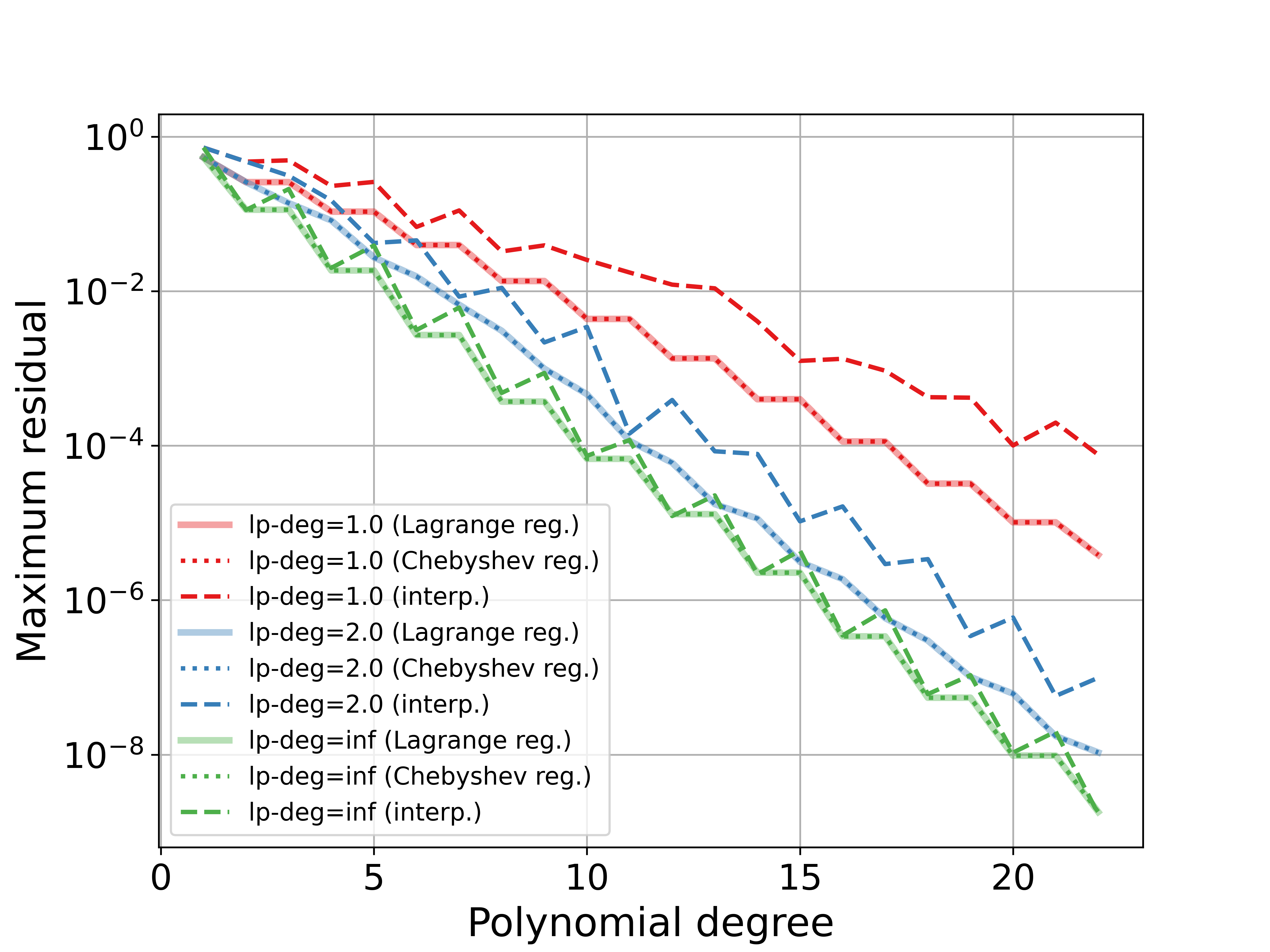}\\
\footnotesize{Legendre} 2D & \footnotesize{Legendre 3D}\\
\vspace{-3pt}
\includegraphics[width=0.44\textwidth]{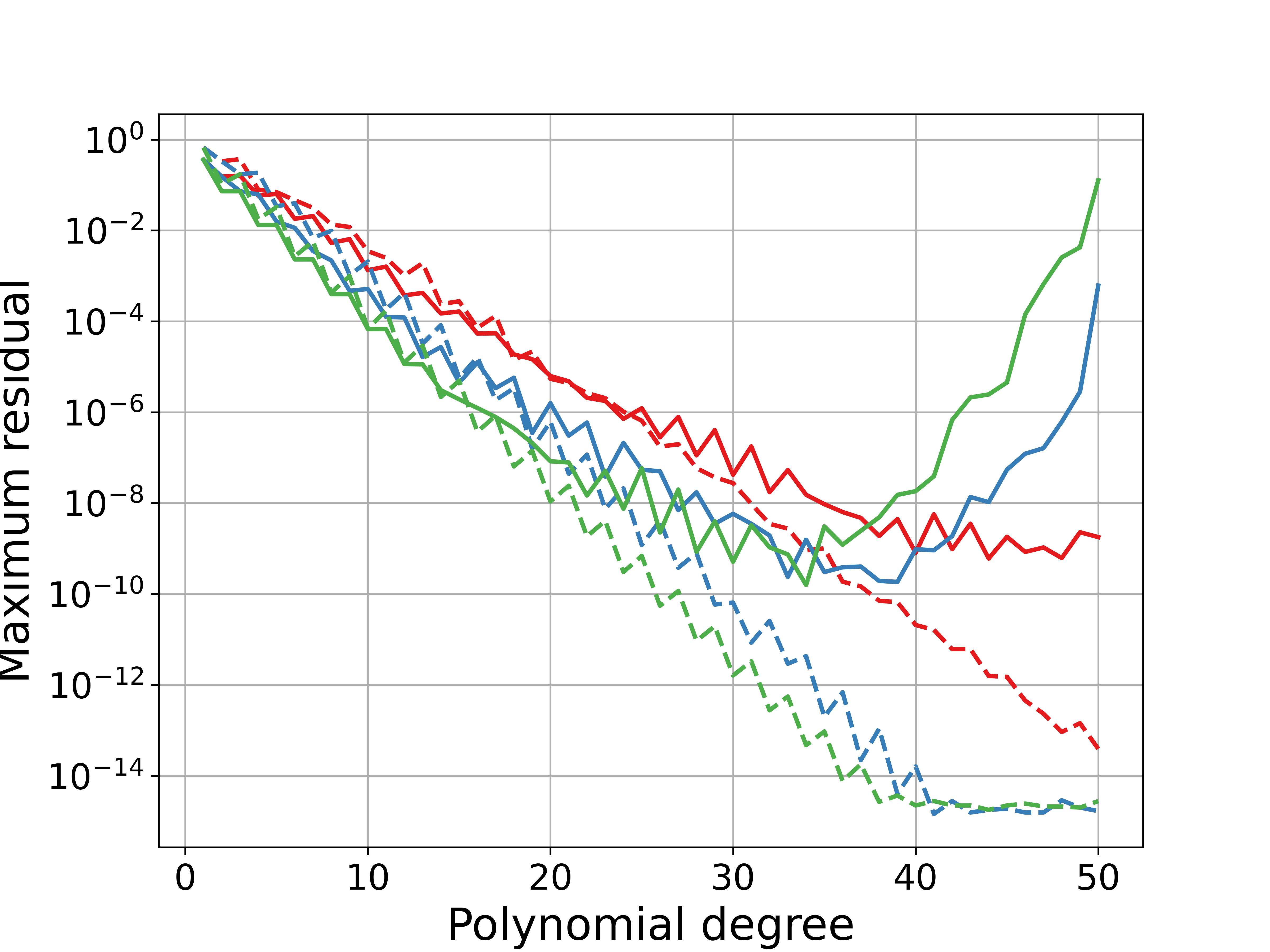}
& \includegraphics[width=0.44\textwidth]{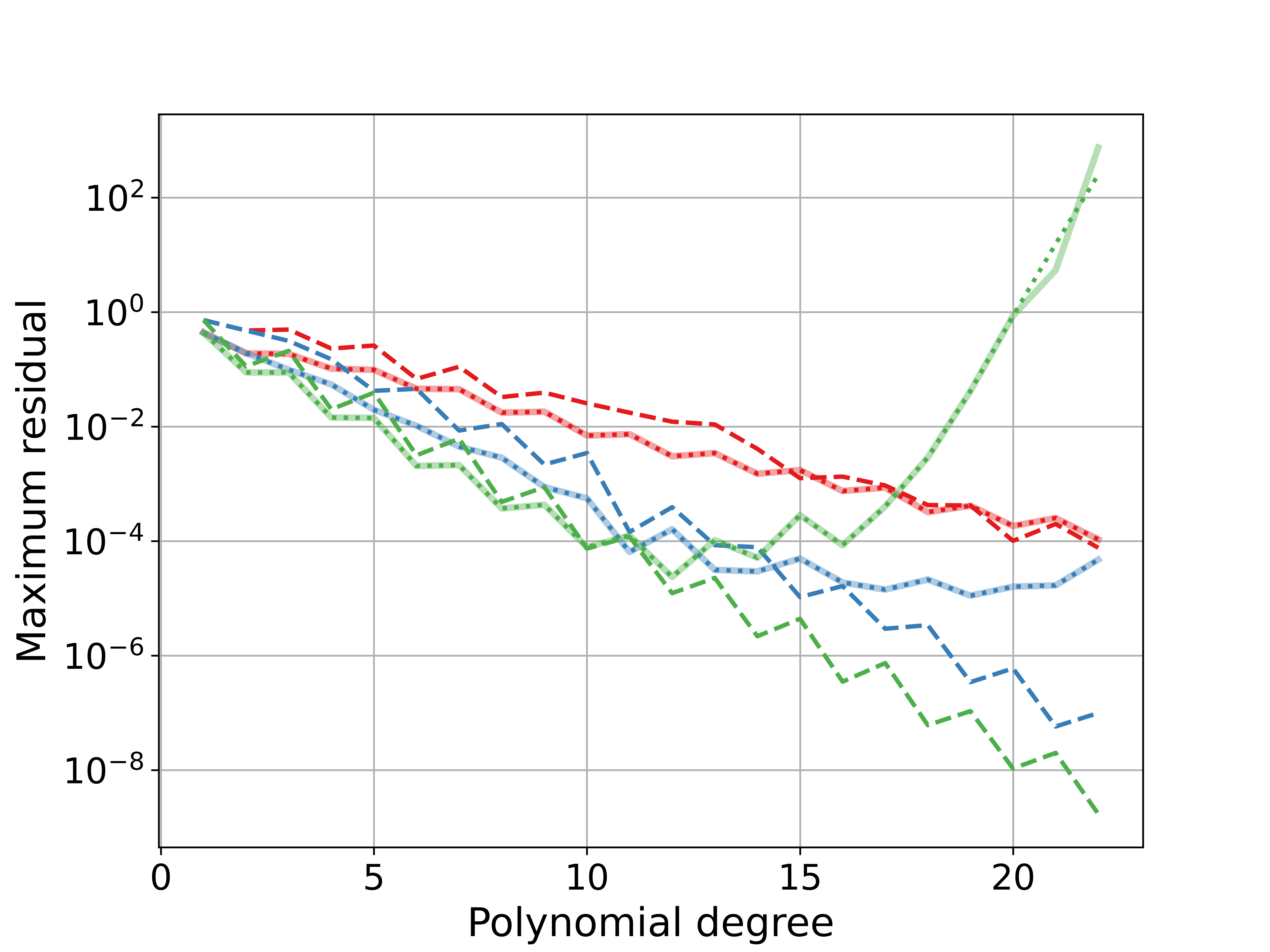}\\
\footnotesize{Random 2D} & \footnotesize{Random 3D} \\
\vspace{-3pt}
\includegraphics[width=0.44\textwidth]{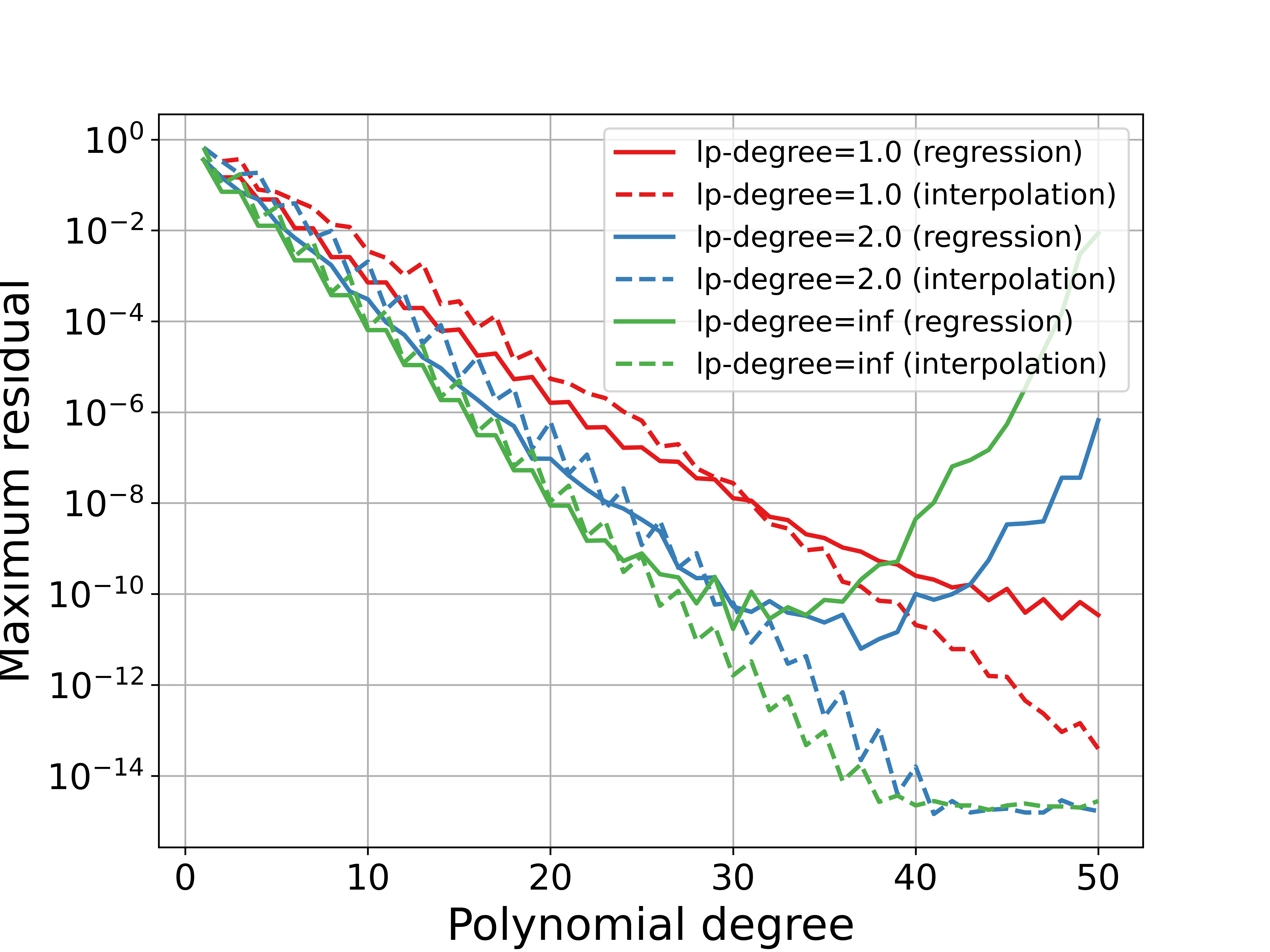}
& \includegraphics[width=0.44\textwidth]{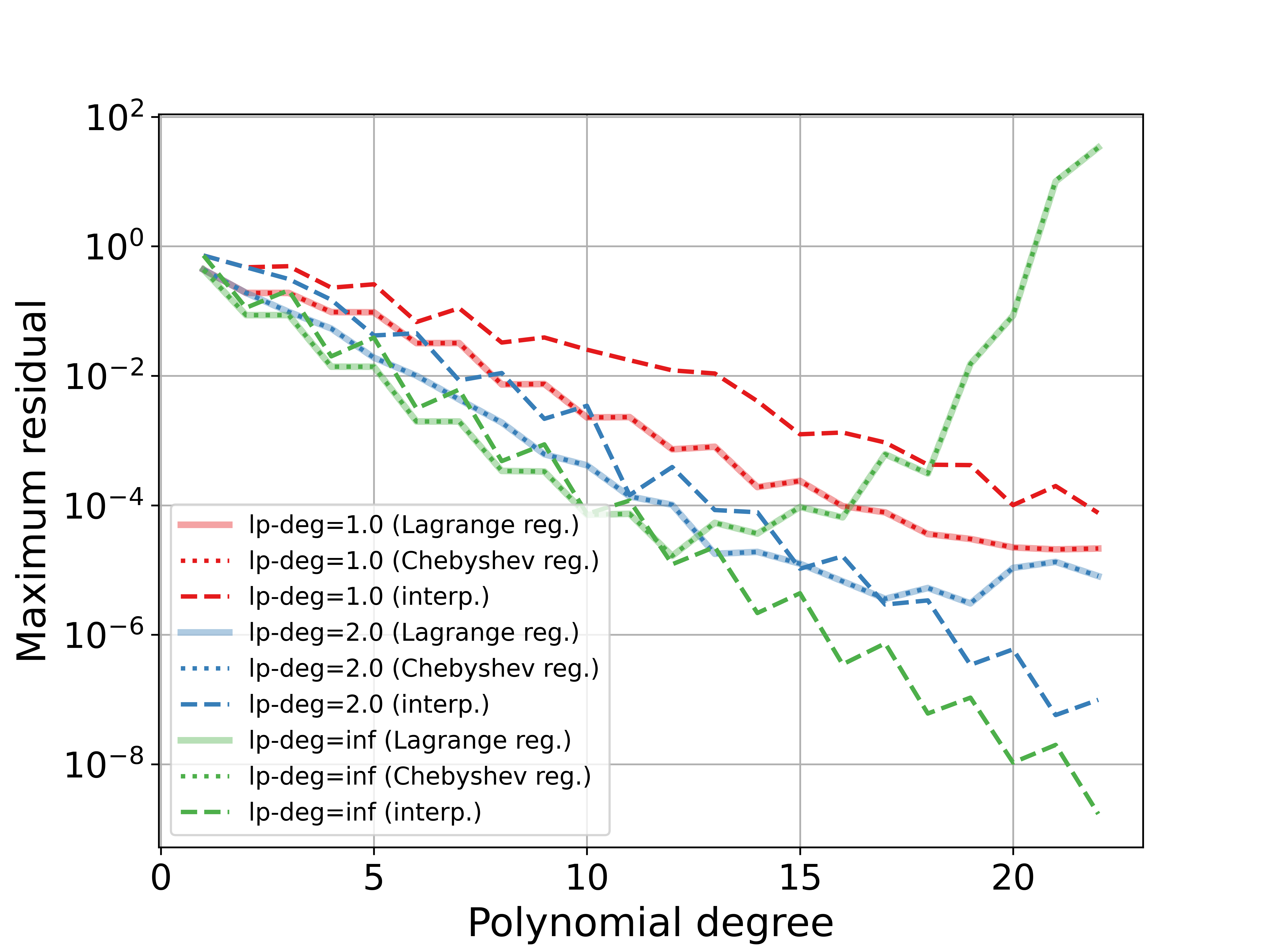}\\
\footnotesize{Sobol 2D} & \footnotesize{Sobol 3D}
\end{tabular}
\vspace{-5pt}
\caption{Newton-Lagrange and Chebyshev $l_p$-regression of the Runge function $f(x) = 1/(1 +\|x\|^2)$ in dimensions $m=2,3$. \label{fig:FIT1}}
\end{figure}

\begin{figure}[ht!]
\center
\begin{tabular}{cc}
\includegraphics[width=0.44\textwidth]{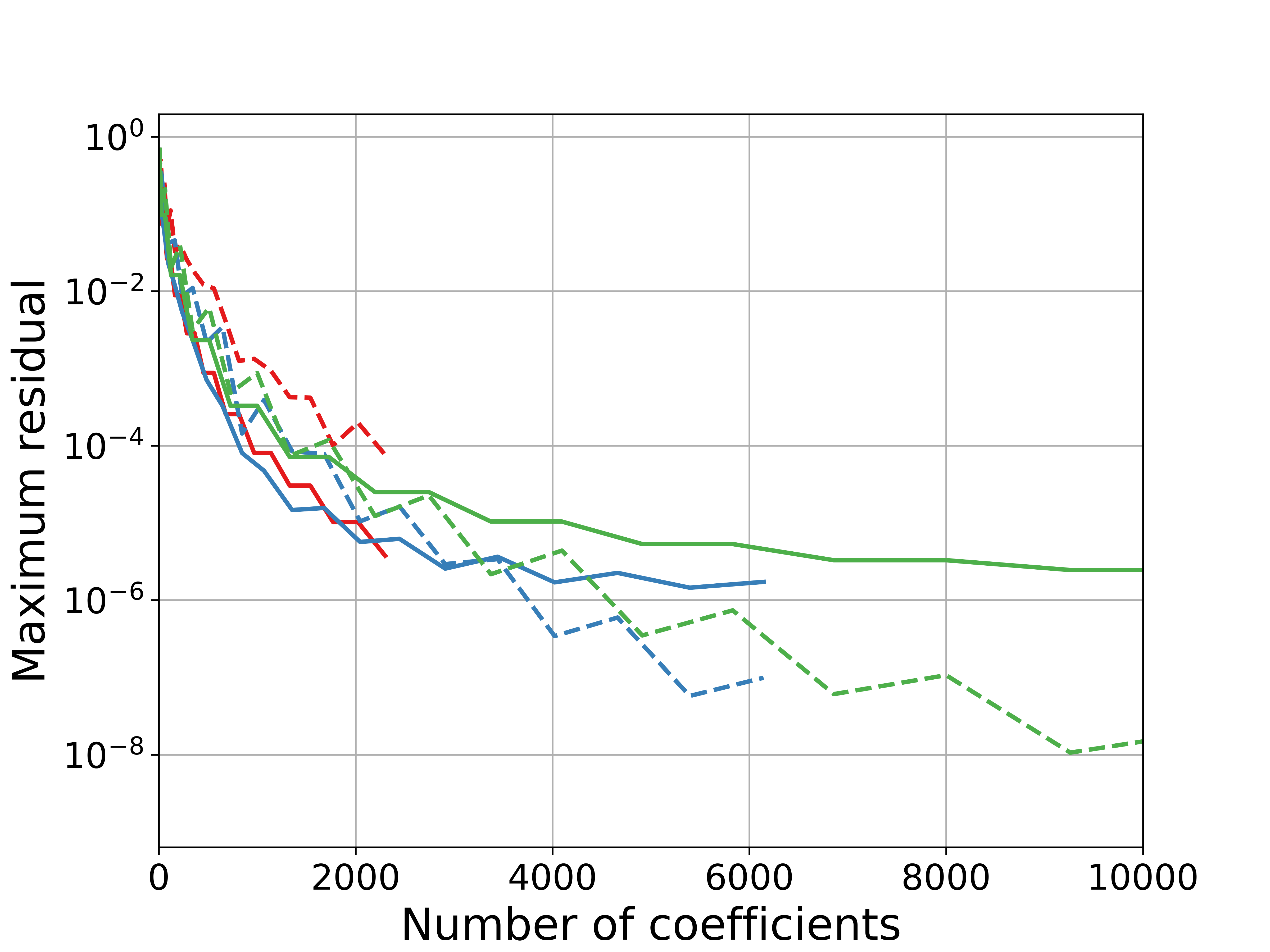}
& \includegraphics[width=0.44\textwidth]{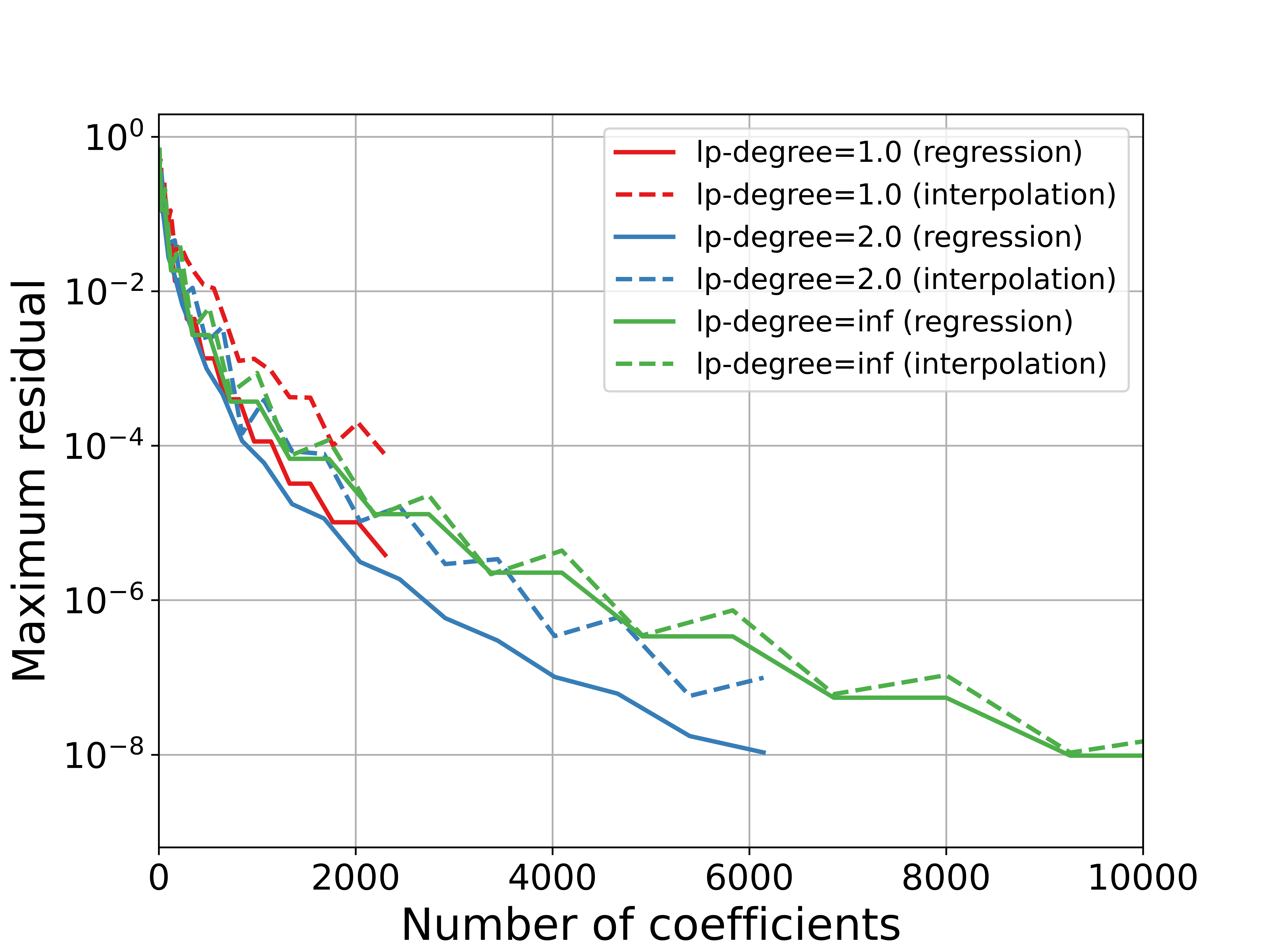}\\
\footnotesize{Equispaced 3D} & \footnotesize{Legendre 3D} \\
\includegraphics[width=0.44\textwidth]{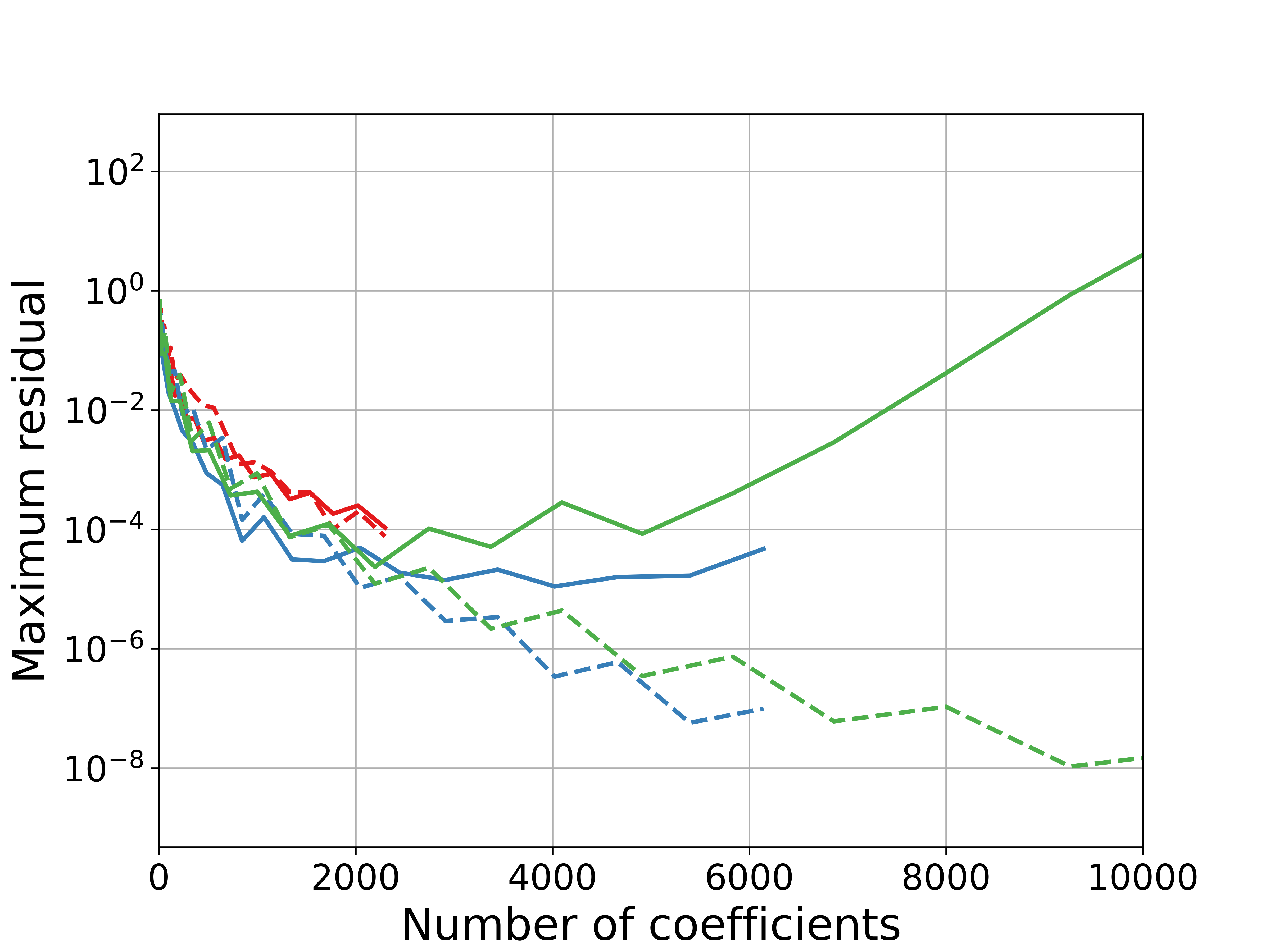}
& \includegraphics[width=0.44\textwidth]{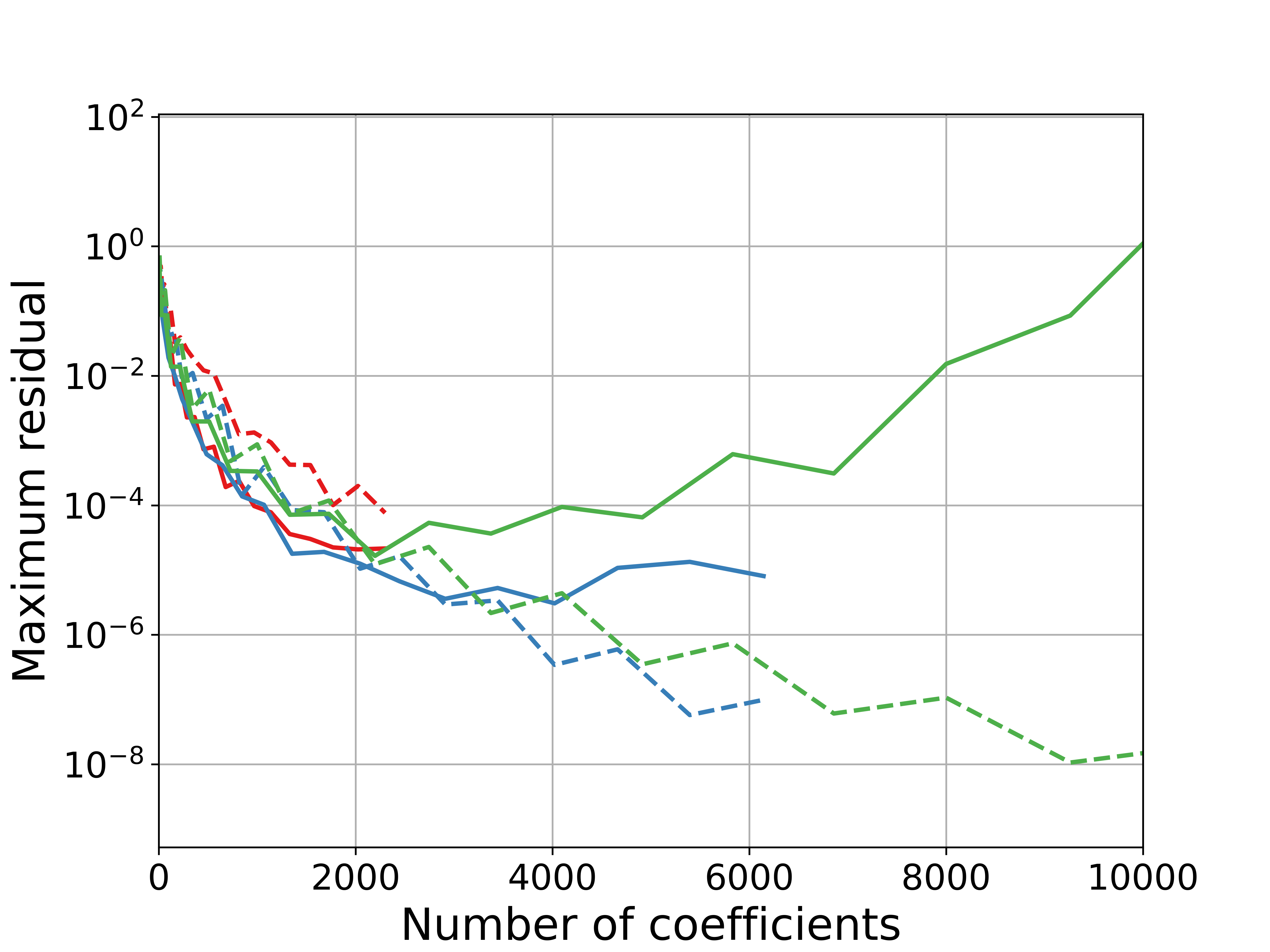}\\
\footnotesize{Random 3D} & \footnotesize{Sobol 3D}
\end{tabular}
\caption{Newton-Lagrange and Chebyshev $l_p$-regression of the Runge function $f(x) = 1/(1 +\|x\|^2)$ in equispaced , random data, and for Sobol sequences in dimensions $m=3$. \label{fig:Coeff}}
\end{figure}

As given in Theorem~\ref{theo:APP}, Eq.~\eqref{eq:EST}, and
discussed in Remark~\ref{rem:LEB}, the approximation power of the regression relies on more factors than the Lebesgue constant only. The following experiments aim to give a better overall picture in this regard.


\begin{experiment}[Approximation factor] We measure the approximation factor
  \linebreak
  $\Lambda(P_A) \|S_{A,P}\|_\infty$, $A= A_{m,n,p}$ occurring in the estimate, Eq.~\eqref{eq:EST} from Theorem~\ref{theo:APP} for equispaced grids, random points, tensorial $l_\infty$-Legendre grids, and Sobol \& Halton sequences, consisting of $256$ points in dimension $m=1$, $64^2=4.096$ points in dimension $m=2$, and $24^3=13.824$ points in dimension $m=3$, the same points for each considered instance $n\in \N$, $p=1,2,\infty$.
\label{exp:FAC}
\end{experiment}


The results are plotted in Fig.~\ref{fig:PInv} for dimension $m=3$. As one observes, the Euclidean and the total degree yield similar approximation factors than the maximum  degree for equispaced data. All choices deliver small approximation factors for the Legendre points and the $l_2,l_\infty$-degree choices perform compatibly for
the other distributions, with the $l_\infty$-regression running into instability first and the $l_1$-degree factors growing less. Contrary to Experiment~\ref{exp:LEB}, the results indicate  the Euclidean and total degree to be the better choices. Next we investigate the conditioning of the corresponding regression matrices.

\begin{figure}[ht!]
     \centering
     \begin{tabular}{cc}
       \includegraphics[width=0.44\textwidth]{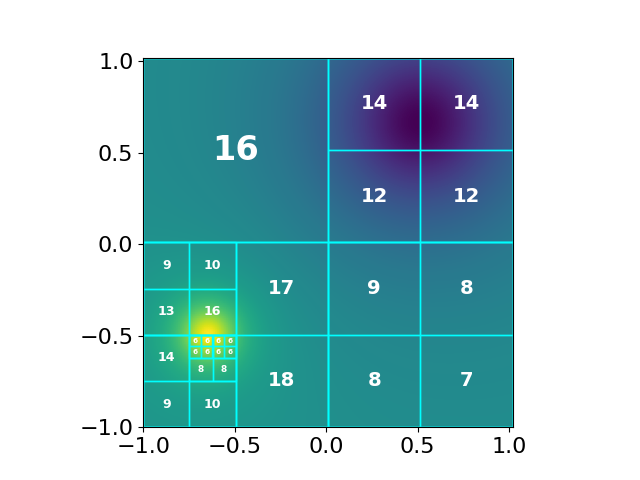}
       & \includegraphics[width=0.44\textwidth]{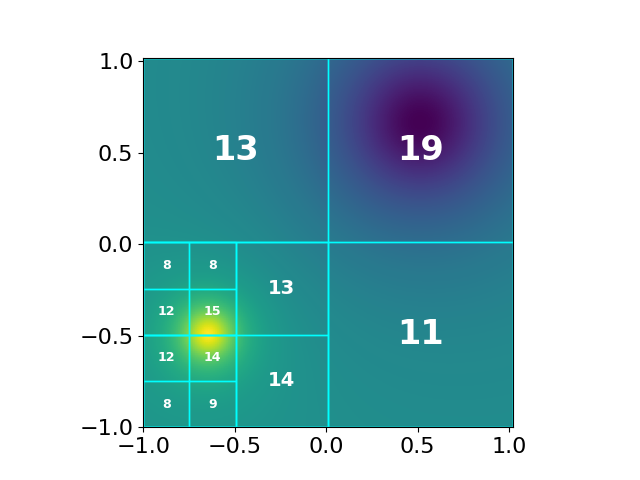}\\
       \footnotesize{Combined Runge function $l_1$-degree} & \footnotesize{Combined Runge function $l_2$-degree}\\
       \includegraphics[width=0.44\textwidth]{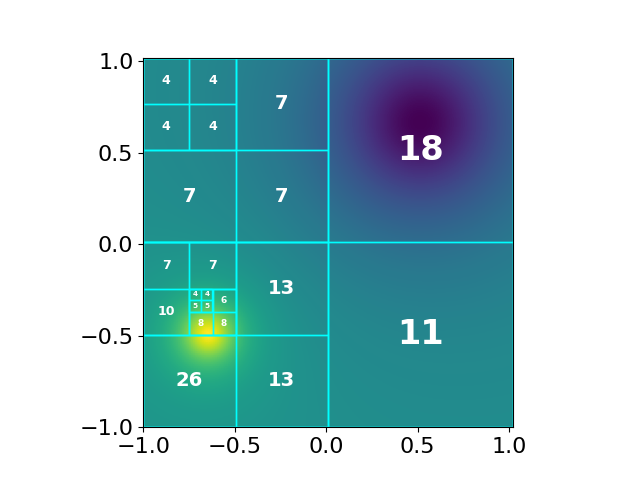}
       & \includegraphics[width=0.44\textwidth]{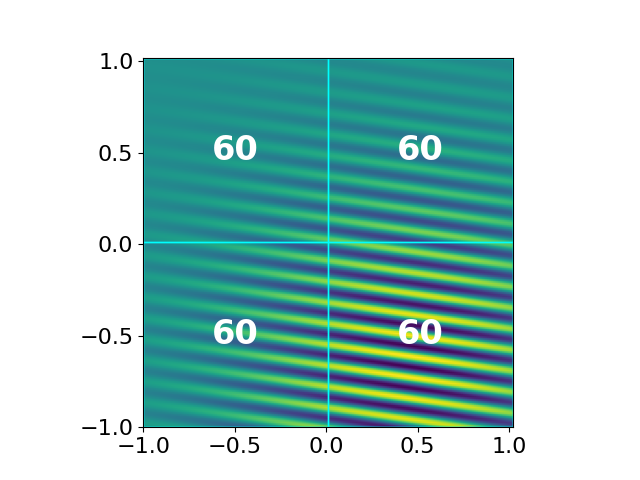}\\
       \footnotesize{Combined Runge function $l_\infty$-degree} & \footnotesize{Gaussian-sine function, all $l_p$-degrees}\\
     \includegraphics[width=0.44\textwidth]{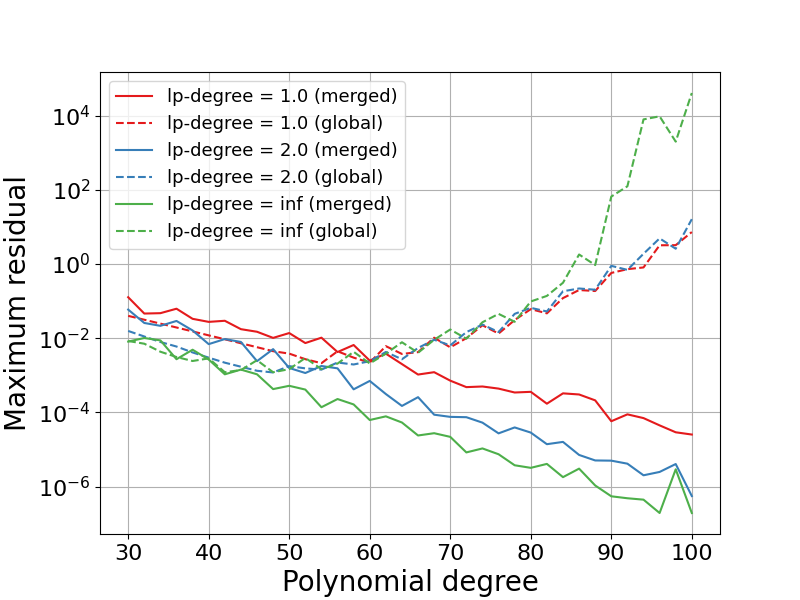}
     & \includegraphics[width=0.44\textwidth]{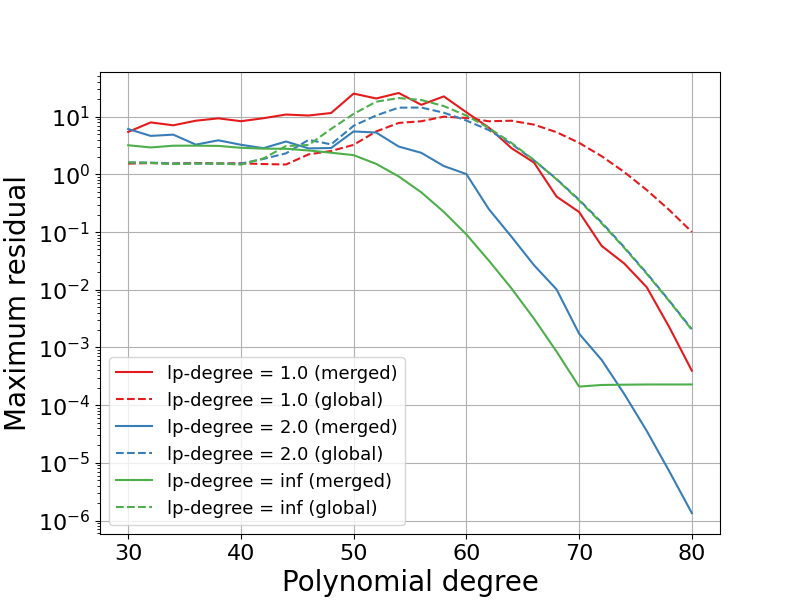}\\
     \footnotesize{Combined Runge function} & \footnotesize{Gaussian-sine function}
     \end{tabular}
        \caption{Adaptive regression: The domain decomposition proposed by the oracle for $\mathrm{F}1,\mathrm{F}2$  and an error tolerance of $10^{-7}$ (top).
        Approximation errors for global and merged regressions for $l_p$-degrees $p=1,2,\infty$ (bottom).\label{Fig:adaptive_decompostion}}
\end{figure}

\begin{experiment}[Condition numbers] We keep the experimental design of Experiment~\ref{exp:FAC} and measure the condition numbers
$$ \cond(R_A) \,, \quad A =A_{m,n,p}$$
of the regression matrices, Eq.~\eqref{eq:REG} for the corresponding setups.
\end{experiment}

The results are illustrated in Fig.~\ref{fig:Cond}. As expected all $l_p$-degree choices result in well-conditioned matrices in case the data points are given by
tensorial Legendre grids. For random, Sobol \& Halton sequences the Euclidean $l_2$-regression resists the instability phenomenon better than  $l_\infty$-regression, but total $l_1$-regression certainly results in the most stable approach. Most crucially, for equispaced data, the Euclidean and total $l_1$-degree choices result in a compatible conditioning of the regression matrices $R_{A_{m,n,p},P}$, $p=1,2,\infty$, respectively, being superior compared to choosing the maximum degree. However, the approximation power of the total degree regression is limited by the slower rate, Eq.~\eqref{Rate}, comapred the Euclidean degree.
We directly compare the potential $l_p$-degree choices to answer how the so far considered factors impact the regression schemes in total.

\begin{experiment}[$l_p$-degree regression] We directly evaluate the $l_p$-regression performance by fitting the Runge function $f(x)=1/(1+\|x\|^2)$, being sampled in equispaced grids, random points, tensorial $l_\infty$-Legendre grids, and Sobol \& Halton sequences, as in Experiment~\ref{exp:FAC}. The approximation errors are measured by generating $1000$ random points $P'$, once for each dimension $m=1,2,3$, and reporting the \emph{maximum residual}
$\max_{x \in P'}|f(x) - Q_{f,A,P}(x)|$, $A=A_{m,n,p}$.
\label{exp:REG}
\end{experiment}

Because regression for Halton sequences performs compatibly as for Sobol sequences, these results were skipped in the plots of Fig.~\ref{fig:FIT1}. Instead, we included the results of the classic Chebyshev regression in dimension $m=3$, being indistinguishable from our proposed Newton-Lagrange regression as initially claimed. The corresponding dashed lines show the approximation errors reached by $l_p$-degree interpolation, which were validated to reach the optimal Trefethen rates, Eq.~\eqref{Rate}, in our prior work \cite{MIP}.

As predicted from Eq.~\eqref{Rate}, the Euclidean and maximum degree regression reach compatibly faster approximation rates than the total $l_1$-degree regression, resulting in several orders of magnitude smaller approximation errors. However, the Euclidean degree regression is significantly more stable than the maximum degree regression, especially for random data and Sobol sequences in dimension $m=3$, Fig~\ref{fig:FIT1}.
It is surprising that all $l_p$-regressions in the tensorial Legendre nodes reach better
accuracy than $l_p$-interpolation in the (non-tensorial) Chebyshev-Lobatto grids $P_{A_{m,n,p}}\subseteq \Omega$, Fig.~\ref{fig:FIT1}.

The superiority of the Euclidean degree over maximum degree becomes even more evident, when comparing the approximation errors with respect to the number of coefficients of the regressor, given in Fig.~\ref{fig:Coeff} for dimension $m=3$. While the Euclidean degree reaches a compatible accuracy as the total degree with the same number of coefficients for equispaced data, it improves by $1\sim 2$ orders of magnitude for the other data distributions. Moreover, in combination with Fig.~\ref{fig:FIT1}, we deduce that the total degree accuracy
comes with the cost of involving monomials with higher degree than required by Euclidean degree. Consequently, $l_2$-regression delivers simpler (lower degree) polynomial models than $l_1$-regression with the same accuracy.

In summary, in most of the investigated factors, implicitly or directly impacting the regression performance, the Euclidean degree regression performed compatibly to or better than the other degree regressions. We conclude that when seeking for a universal scheme that will perform best for general regression tasks, Euclidean degree is the pivotal choice.

Nevertheless, the Euclidean degree regression is limited by the initially mentioned factors, Section~\ref{sec:limits}. In order to investigate how to resist these limitations further, we evaluate the adaptive domain decomposition from Section~\ref{sec:AD}.

\begin{experiment}[Adaptive regression]
We  execute the adaptive regression from Section~\ref{sec:AD} for three show cases:
\begin{enumerate}
  \item[F1)] The combined Runge function $$f(x) = 1 / (1 + 50||x-x_1||^2) - 1 / (1 + 5||x-x_2||^2)\,,$$
  where $x_1 = -(0.65,0.5)$ and $x_2 = (0.5,0.65)$, in dimension $m =2$.
  \item[F2)] The Gaussian-sine function  $$f(x) = \exp(-\|x-x_0\|^2)\big(\cos(\pi k(\eta\cdot x)) + \sin(\pi k (\eta\cdot x))\big)\,,$$
  where $x_0 = -(0.45,-0.65)$ $\eta = (0.10,0.70)$, and $k=25 \in \N$.
  \item[F3)] The piecewise polynomial in dimension $m=3$, given by
  $$ f(x_1,x_2,x_3) = \left\{\begin{array}{ll}
  x_1x_2^2x_3^3 & \,, x_1,x_2,x_3 \geq 0 \\
  0 & \,, \text{otherwise}
  \end{array}\right.\,.$$
\end{enumerate}
  $F1),F2)$ are sampled on an equispaced grid $G$ of resolution $200 \times 200$ and approximation errors are measured on $2.500$ random points. $F3)$ is sampled on on an equispaced grid $G$ of resolution $15 \times 15 \times 15$ and approximation errors are measured on $1.000$ random points.
\end{experiment}
\begin{table}[ht]
  \centering
\begin{tabular}{lccccc}
Function $\mathrm{F}1)$ & $l_p$-degree  & regressors  & $|C|$ & total $\mathrm{cf}$ & $\mathrm{cf}$ merging $n=50$/$100$\\
\hline
& Total         & $10$ & $2.482$ & $16$ &  ${\bf 30}$  / ${\bf 8}$\\
& Euclidean     & $13$ & ${\bf 2.322}$  & ${\bf 17}$ & $20$ / $5$\\
& Maximum       & $31$ & $3.214$ & $24$  & $15$ / $4$\\
& & & & \\
Function $\mathrm{F}2)$ & $l_p$-degree  & regressors  & $|C|$ & total $\mathrm{cf}$ &  $\mathrm{cf}$ merging $n=50$/$80$\\
\hline
& Total         & $4$ & $7.564 $ & $5$ & ${\bf 30}$  / ${\bf 12}$\\
& Euclidean     & $4$ & $11.532$  & $3$ & $20$ / $8$\\
& Maximum       & $4$ & $14.484$ & $3$  & $15$ / $6$\\
\end{tabular}
\caption{Efficiency in terms of number of coefficients $|C|$ and compression factors $\mathrm{cf}$ of the adaptive regression and the global merging for $l_p$-degrees, $p=1,2,\infty$. \label{tab:AR}}
\end{table}
For $\mathrm{F}1),\mathrm{F}2)$,  Fig.~\ref{Fig:adaptive_decompostion}~(top) shows the adaptive domain decomposition proposed by the oracle with respect to the chosen $l_p$-degree, Definition~\ref{def:orac}, when demanding an error tolerance of $10^{-7}$ (for both functions). While all $l_p$-degree choices delivered the same decomposition for $\mathrm{F}2)$ it is only reported once.

In the bottum the maximum residuals for the global $l_p$-regressions, Eq.~\eqref{eq:REG}, and the merged $l_p$-degree polynomials, Definition~\ref{def:MERGE}, are plotted.
In case of $\mathrm{F}1)$, for all $l_p$-degree choices, the global regressions converge to an approximation error of $10^{-3}$ at best and run into instability for degrees $n \geq 50$,
reflecting the rapidly growing condition numbers in that range.
In contrast, the globally merged polynomials with respect to $l_2,l_\infty$-degree achieve approximations of up to $10^{-6}$ accuracy without showing any unstable behaviour.

In case of $\mathrm{F}2)$, global $l_1$-degree regression reaches approximation error of $10^{-1}$, while $l_2,l_\infty$-degree choices perform indistinguishable with approximation error of
$10^{-3}$ for $n=80$. In contrast, Euclidean degree merging outperforms the total and maximum degree choices by reaching $10^{-6}$ approximation accuracy compared to $10^{-4}$ for the latter choices.

Table~\ref{tab:AR} reports
the  number of individual regressors,  the  number of coefficients $|C|$ (the total sum) and the compression factor  $\mathrm{cf}= 40.000 / |C|$ for the decomposition and for the globally merged polynomials with $n=50$, $n=80$, $p=1,2,\infty$.
The superiority of the Euclidean degree decomposition in terms of compression becomes obvious for $\mathrm{F}1)$, while $l_1$-degree performs better in the case of $\mathrm{F}2)$, and when used for global merging.
However, given that especially in the $\mathrm{F}2)$-case, the approximation accuracy of $l_1$-degree regression is by far non-compatible with the Euclidean choice, which is, thus, the most efficient choice.

Fig.~\ref{fig:3d} illustrates the results of the function $\mathrm{F}3)$ and gives further evidence on that conclusion. While $\mathrm{F}3$ is sampled on a $15 \times 15 \times 15$ equidistant grid all global regressions are limited by a maximum degree $n=14$. However, best approximation, $10^{-2}$ error, is reached for degree $n=10$ and all
$l_p$-regressions start to become unstable for $n >10$.  In contrast, global merging is is applicable and stabel for higher degrees, improving up to
$10^{-4}$ approximation error for $l_2,l_\infty$-degree choices. However, Euclidean degree delivers once more better compression performance than choosing maximum degree for adaptivle decomposing the fitting task with error tolerance $10^{-10}$.

\begin{figure}[t!]
    \centering
    \begin{subfigure}[t]{.5\textwidth}
        \centering
    \includegraphics[width=1.0\textwidth]{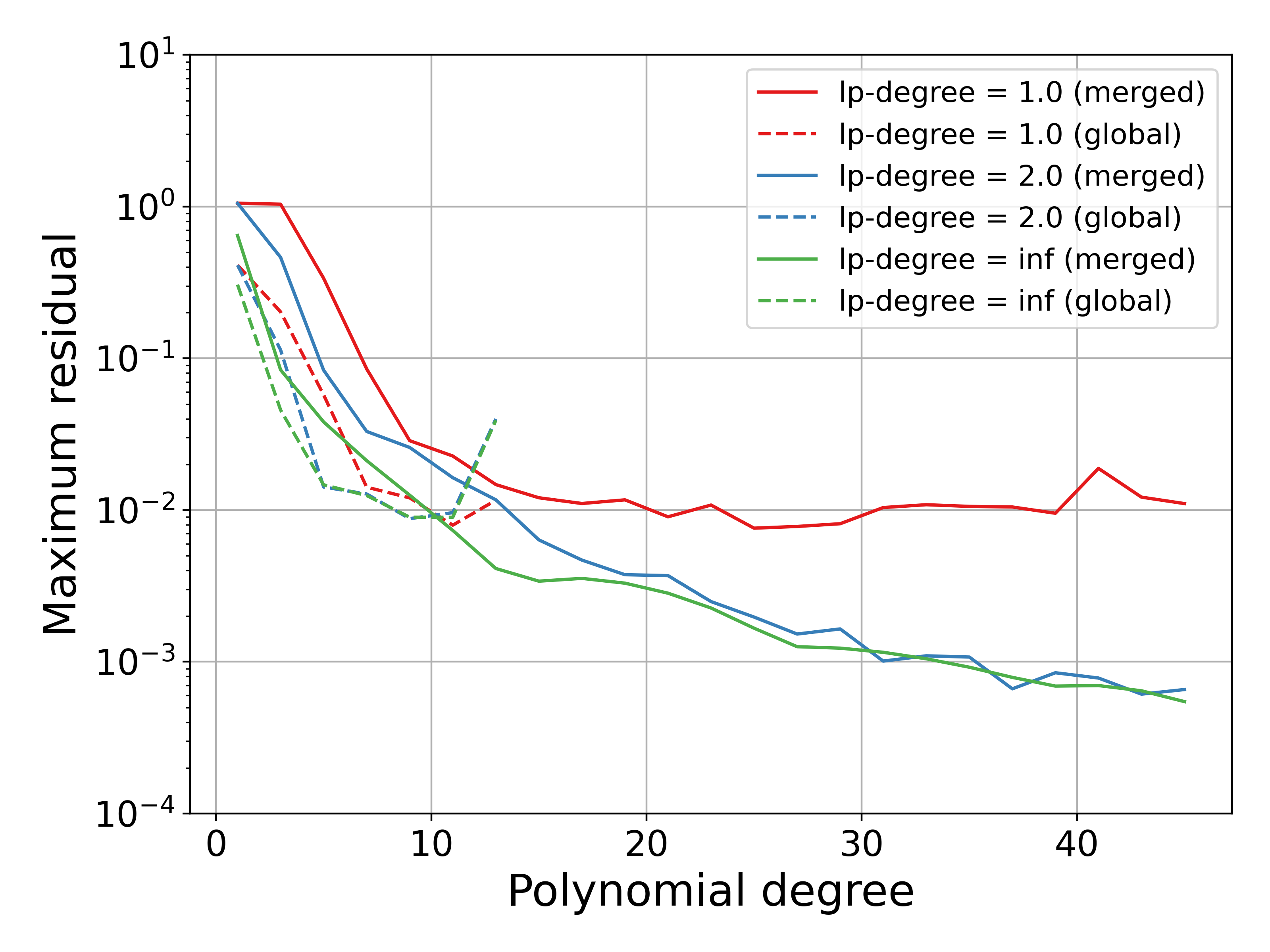}
    \end{subfigure}%
    \begin{subfigure}[t]{.5\textwidth}
        \centering
        \begin{tabular}{lcc}
     $l_p$-degree  & regressors  & $|C|$ \\
        \hline
         Total         & $36$ & $2.626$  \\
         Euclidean     & $57$ & ${\bf 2.373}$  \\
         Maximum       & $8$  & $2.617$ \\
         & & \\
         & & \\
         & & \\
         & & \\
         & & \\
         & & \\
         & & \\
         & & \\
        \end{tabular}
    \end{subfigure}
    \vspace{-75pt}
    \caption{Adaptive regression: Approximation errors for $l_p$-degree regression based on the domain decomposition proposed by the oracle
    for $\mathrm{F}3)$ and an error tolerance of $10^{-10}$, $p=1,2,\infty$ (left). Compression efficiency in terms of total number of coefficients $|C|$ (right). \label{fig:3d}}
\end{figure}

We summerize these and our prior findings by giving the following concluding thoughts.

\section{Conclusion}
We have proposed general $l_p$-degree Newton-Lagrange regression
\linebreak schemes resting on our former work, addressing multivariate interpolation tasks in non-
\linebreak tensorial nodes.
We theoretically argued and empirically demonstrated that in most aspects the choice of the Euclidean $l_2$-degree is superior to the total or maximum degree choices
for regression tasks of regular Trefethen functions. We further suggest to use the concept to extend the limitations for fast function approximations by applying an adaptive domain decompositions approach and afterwards globally merging the individual regressors due to $l_p$-degree interpolation. As the results suggest we presume that the presented concepts can impact
established strategies addressing multivariate function approximation tasks, as initially discussed.

%

\bibliographystyle{siamplain}
\bibliography{/Users/admin_hecht93/Documents/REF/Ref.bib}

\end{document}

%% file: siam_shared.tex
\usepackage{lipsum}
\usepackage{amsfonts}
\usepackage{graphicx}
\usepackage{epstopdf}
\ifpdf
  \DeclareGraphicsExtensions{.eps,.pdf,.png,.jpg}
\else
  \DeclareGraphicsExtensions{.eps}
\fi

\newsiamremark{remark}{Remark}
\newsiamremark{hypothesis}{Hypothesis}
\crefname{hypothesis}{Hypothesis}{Hypotheses}
\newsiamthm{claim}{Claim}

\headers{Polynomial Regression of Euclidean Degree}{S. K. Thekke Veettil, Y. Zheng, U. H. Acosta, D. Wicaksono, and M. Hecht}

\title{Multivariate Polynomial Regression of Euclidean Degree Extends the Stability for Fast Approximations of Trefethen Functions\thanks{
\funding{This work was partially funded by the Center of Advanced Systems Understanding (CASUS), financed by Germany's Federal Ministry of Education and Research (BMBF) and by the Saxon Ministry for Science, Culture and Tourism (SMWK) with tax funds on the basis of the budget approved by the Saxon State Parliament.}}}

\author{Sachin Krishnan Thekke Veettil\thanks{Technische Universit\"{a}t Dresden,
Faculty of Computer Science, Dresden, Germany.\\
\indent \indent Max Planck Institute of Molecular Cell Biology and Genetics, Dresden, Germany.\\
\indent \indent Center for Systems Biology Dresden, Dresden Germany
}
\and Yuxi Zheng
\and Uwe Hernandez Acosta\footnotemark[3]
\and Damar Wicaksono \thanks{Center for Advanced Systems Understanding (CASUS), G\"{o}rlitz, Germany.}
\and Michael Hecht\footnotemark[3] \thanks{Corresponding author. Email: m.hecht@hzdr.de}}

\usepackage{amsopn}


%% file: REG.bbl
\begin{thebibliography}{10}

\bibitem{baltensperger1999exponential}
{\sc R.~Baltensperger, J.-P. Berrut, and B.~No{\"e}l}, {\em Exponential
  convergence of a linear rational interpolant between transformed chebyshev
  points}, Mathematics of Computation, 68 (1999), pp.~1109--1120.

\bibitem{ben2003}
{\sc A.~Ben-Israel and T.~N. Greville}, {\em Generalized inverses: theory and
  applications}, vol.~15, Springer Science \& Business Media, 2003.

\bibitem{berzins2007adaptive}
{\sc M.~Berzins}, {\em Adaptive polynomial interpolation on evenly spaced
  meshes}, Siam Review, 49 (2007), pp.~604--627.

\bibitem{converse}
{\sc L.~Bos and N.~Levenberg}, {\em Bernstein--{W}alsh theory associated to
  convex bodies and applications to multivariate approximation theory},
  Computational Methods and Function Theory, 18 (2018), pp.~361--388.

\bibitem{boyd1992defeating}
{\sc J.~P. Boyd}, {\em Defeating the {R}unge phenomenon for equispaced
  polynomial interpolation via {T}ikhonov regularization}, Applied mathematics
  letters, 5 (1992), pp.~57--59.

\bibitem{boyd2011exponentially}
{\sc J.~P. Boyd and J.~R. Ong}, {\em Exponentially-convergent strategies for
  defeating the runge phenomenon for the approximation of non-periodic
  functions, part two: Multi-interval polynomial schemes and multidomain
  chebyshev interpolation}, Applied numerical mathematics, 61 (2011),
  pp.~460--472.

\bibitem{boyd2009divergence}
{\sc J.~P. Boyd and F.~Xu}, {\em Divergence ({R}unge phenomenon) for
  least-squares polynomial approximation on an equispaced grid and
  {M}ock--{C}hebyshev subset interpolation}, Applied Mathematics and
  Computation, 210 (2009), pp.~158--168.

\bibitem{brutman2}
{\sc L.~Brutman}, {\em Lebesgue functions for polynomial interpolation -- a
  survey}, Annals of Numerical Mathematics, 4 (1996), pp.~111--128.

\bibitem{cohen2}
{\sc A.~Chkifa, A.~Cohen, and C.~Schwab}, {\em High-dimensional adaptive sparse
  polynomial interpolation and applications to parametric pdes}, Foundations of
  Computational Mathematics, 14 (2014), pp.~601--633.

\bibitem{cohen3}
{\sc A.~Cohen and G.~Migliorati}, {\em Multivariate approximation in downward
  closed polynomial spaces}, in Contemporary Computational Mathematics-A
  celebration of the 80th birthday of Ian Sloan, Springer, 2018, pp.~233--282.

\bibitem{ConteSamuelDaniel2017Ena:}
{\sc S.~D. Conte and C.~De~Boor}, {\em Elementary numerical analysis : an
  algorithmic approach / S.D. Conte, Carl de Boor.}, Classics in applied
  mathematics ; 78, Society for Industrial and Applied Mathematics SIAM, 3600
  Market Street, Floor 6, Philadelphia, PA 19104, Philadelphia, Pennsylvania,
  3d ed., classics edition.~ed., 2017.

\bibitem{Boor:BS}
{\sc C.~De~Boor}, {\em On calculating with {B}-splines}, Journal of
  Approximation theory, 6 (1972), pp.~50--62.

\bibitem{Boor:tensorSP}
{\sc C.~De~Boor}, {\em Efficient computer manipulation of tensor products.},
  tech. report, Wisconsin Univ. Madison Mathematics Research Center, 1977.

\bibitem{Boor:SP}
{\sc C.~De~Boor}, {\em A practical guide to splines}, vol.~27,
  {S}pringer-{V}erlag New York, 1978.

\bibitem{Boor:wings}
{\sc C.~De~Boor}, {\em On wings of splines}, in Creative Minds, Charmed Lives:
  {I}nterviews at {I}nstitute for {M}athematical Sciences, National University
  of Singapore, World Scientific, 2010, pp.~50--57.

\bibitem{boorapp}
{\sc C.~De~Boor and K.~H{\"o}llig}, {\em Approximation power of smooth
  bivariate pp functions}, Mathematische Zeitschrift, 197 (1988), pp.~343--363.

\bibitem{dimarogonas1996vibration}
{\sc A.~D. Dimarogonas}, {\em Vibration for engineers}, Prentice Hall, 1996.

\bibitem{driscoll2002interpolation}
{\sc T.~A. Driscoll and B.~Fornberg}, {\em Interpolation in the limit of
  increasingly flat radial basis functions}, Computers \& Mathematics with
  Applications, 43 (2002), pp.~413--422.

\bibitem{chebfun}
{\sc T.~A. Driscoll, N.~Hale, and L.~N. Trefethen}, {\em Chebfun guide},
  Pafnuty Publications, Oxford,  (2014).

\bibitem{floater}
{\sc M.~S. Floater and K.~Hormann}, {\em Barycentric rational interpolation
  with no poles and high rates of approximation}, Numerische Mathematik, 107
  (2007), pp.~315--331.

\bibitem{fornberg2008stable}
{\sc B.~Fornberg and C.~Piret}, {\em A stable algorithm for flat radial basis
  functions on a sphere}, SIAM Journal on Scientific Computing, 30 (2008),
  pp.~60--80.

\bibitem{fornberg2004stable}
{\sc B.~Fornberg and G.~Wright}, {\em Stable computation of multiquadric
  interpolants for all values of the shape parameter}, Computers \& Mathematics
  with Applications, 48 (2004), pp.~853--867.

\bibitem{fornberg2007runge}
{\sc B.~Fornberg and J.~Zuev}, {\em The runge phenomenon and spatially variable
  shape parameters in rbf interpolation}, Computers \& Mathematics with
  Applications, 54 (2007), pp.~379--398.

\bibitem{gautschi}
{\sc W.~Gautschi}, {\em Numerical analysis}, Springer Science \& Business
  Media, 2011.

\bibitem{gergonne1974application}
{\sc J.~D. Gergonne}, {\em The application of the method of least squares to
  the interpolation of sequences}, Historia Mathematica, 1 (1974),
  pp.~439--447.

\bibitem{hale2009use}
{\sc N.~Hale}, {\em On the use of conformal maps to speed up numerical
  computations}, PhD thesis, D.Phil.,Oxford University Computing Lab., Oxford,
  2009.

\bibitem{hale2008new}
{\sc N.~Hale and L.~N. Trefethen}, {\em New quadrature formulas from conformal
  maps}, SIAM Journal on Numerical Analysis, 46 (2008), pp.~930--948.

\bibitem{harten1987uniformly}
{\sc A.~Harten, B.~Engquist, S.~Osher, and S.~R. Chakravarthy}, {\em Uniformly
  high order accurate essentially non-oscillatory schemes, iii}, in Upwind and
  high-resolution schemes, Springer, 1987, pp.~218--290.

\bibitem{PIP1}
{\sc M.~Hecht, B.~L. Cheeseman, K.~B. Hoffmann, and I.~F. Sbalzarini}, {\em A
  quadratic-time algorithm for general multivariate polynomial interpolation},
  arXiv preprint arXiv:1710.10846,  (2017).

\bibitem{MIP}
{\sc M.~Hecht, K.~Gonciarz, J.~Michelfeit, V.~Sivkin, and I.~F. Sbalzarini},
  {\em Multivariate interpolation in unisolvent nodes--lifting the curse of
  dimensionality}, arXiv preprint arXiv:2010.10824,  (2020).

\bibitem{PIP2}
{\sc M.~Hecht, K.~B. Hoffmann, B.~L. Cheeseman, and I.~F. Sbalzarini}, {\em
  Multivariate {N}ewton interpolation}, arXiv preprint arXiv:1812.04256,
  (2018).

\bibitem{IEEE}
{\sc M.~Hecht and I.~F. Sbalzarini}, {\em Fast interpolation and {F}ourier
  transform in high-dimensional spaces}, in Intelligent Computing. Proc. 2018
  IEEE Computing Conf., Vol. 2,, K.~Arai, S.~Kapoor, and R.~Bhatia, eds.,
  vol.~857 of Advances in Intelligent Systems and Computing, London, UK, 2018,
  Springer Nature, pp.~53--75.

\bibitem{minterpy}
{\sc U.~Hernandez~Acosta, S.~Krishnan Thekke~Veettil, D.~Wicaksono, and
  M.~Hecht}, {\em {\sc minterpy} - multivariate interpolation in python},
  https://github.com/casus/minterpy/,  (2021).

\bibitem{hewitt1979gibbs}
{\sc E.~Hewitt and R.~E. Hewitt}, {\em The {G}ibbs-{W}ilbraham phenomenon: an
  episode in {F}ourier analysis}, Archive for history of Exact Sciences,
  (1979), pp.~129--160.

\bibitem{leja}
{\sc F.~Leja}, {\em Sur certaines suites li{\'e}es aux ensembles plans et leur
  application {\`a} la repr{\'e}sentation conforme}, in Annales Polonici
  Mathematici, vol.~1, Instytut Matematyczny Polskiej Akademi Nauk, 1957,
  pp.~8--13.

\bibitem{LIP}
{\sc E.~Meijering}, {\em A chronology of interpolation: {F}rom ancient
  astronomy to modern signal and image processing}, Proceedings of the {IEEE},
  90 (2002), pp.~319--342.

\bibitem{platte2011fast}
{\sc R.~B. Platte}, {\em How fast do radial basis function interpolants of
  analytic functions converge?}, IMA journal of numerical analysis, 31 (2011),
  pp.~1578--1597.

\bibitem{platte2005polynomials}
{\sc R.~B. Platte and T.~A. Driscoll}, {\em Polynomials and potential theory
  for gaussian radial basis function interpolation}, SIAM Journal on Numerical
  Analysis, 43 (2005), pp.~750--766.

\bibitem{platte:2011}
{\sc R.~B. Platte, L.~N. Trefethen, and A.~B. Kuijlaars}, {\em Impossibility of
  fast stable approximation of analytic functions from equispaced samples},
  SIAM review, 53 (2011), pp.~308--318.

\bibitem{rakhmanov2007bounds}
{\sc E.~A. Rakhmanov}, {\em Bounds for polynomials with a unit discrete norm},
  Annals of mathematics,  (2007), pp.~55--88.

\bibitem{runge}
{\sc C.~Runge}, {\em {\"U}ber empirische {F}unktionen und die {I}nterpolation
  zwischen {\"a}quidistanten {O}rdinaten}, Zeitschrift f{\"u}r Mathematik und
  Physik, 46 (1901), p.~20.

\bibitem{schonhage1961fehlerfortpflanzung}
{\sc A.~Sch{\"o}nhage}, {\em {F}ehlerfortpflanzung bei {I}nterpolation},
  Numerische Mathematik, 3 (1961), pp.~62--71.

\bibitem{stengle1989chebyshev}
{\sc G.~Stengle}, {\em {C}hebyshev interpolation with approximate nodes of
  unrestricted multiplicity}, Journal of approximation theory, 57 (1989),
  pp.~1--13.

\bibitem{stigler1974gergonne}
{\sc S.~M. Stigler}, {\em {G}ergonne's 1815 paper on the design and analysis of
  polynomial regression experiments}, Historia Mathematica, 1 (1974),
  pp.~431--439.

\bibitem{Stoer}
{\sc J.~Stoer, R.~Bulirsch, R.~H. Bartels, W.~Gautschi, and C.~Witzgall}, {\em
  Introduction to numerical analysis}, Texts in applied mathematics, Springer,
  New York, 2002.

\bibitem{tadmor2007filters}
{\sc E.~Tadmor}, {\em Filters, mollifiers and the computation of the gibbs
  phenomenon}, Acta Numerica, 16 (2007), pp.~305--378.

\bibitem{tee2006rational}
{\sc T.~W. Tee and L.~N. Trefethen}, {\em A rational spectral collocation
  method with adaptively transformed {C}hebyshev grid points}, SIAM Journal on
  Scientific Computing, 28 (2006), pp.~1798--1811.

\bibitem{Lloyd2}
{\sc L.~N. Trefethen}, {\em Multivariate polynomial approximation in the
  hypercube}, Proceedings of the American Mathematical Society, 145 (2017),
  pp.~4837--4844.

\bibitem{Lloyd}
{\sc L.~N. Trefethen}, {\em Approximation theory and approximation practice},
  vol.~164, SIAM, 2019.

\bibitem{Lloyd_Num}
{\sc L.~N. Trefethen and D.~Bau~III}, {\em Numerical linear algebra}, vol.~50,
  SIAM, 1997.

\bibitem{turetskii1940bounding}
{\sc A.~Turetskii}, {\em The bounding of polynomials prescribed at equally
  distributed points}, in Proc. Pedag. Inst. Vitebsk, vol.~3, 1940,
  pp.~117--127.

\bibitem{wang2010rational}
{\sc Q.~Wang, P.~Moin, and G.~Iaccarino}, {\em A rational interpolation scheme
  with superpolynomial rate of convergence}, SIAM Journal on Numerical
  Analysis, 47 (2010), pp.~4073--4097.

\bibitem{wendland2005scattered}
{\sc H.~Wendland}, {\em Scattered data approximation}, Cambridge University
  Press,  (2005).

\end{thebibliography}
